\documentclass[reqno,12pt,letterpaper]{amsart}
\usepackage{amsmath,amssymb,amsthm,graphicx,mathrsfs,url}
\usepackage[usenames,dvipsnames]{color}
\usepackage[colorlinks=true,linkcolor=Red,citecolor=Green]{hyperref}
\usepackage{amsxtra}
\usepackage{epstopdf}

\usepackage{amsmath}
\usepackage{mathtools}
\usepackage{amsfonts}
\usepackage{mathrsfs}
\usepackage{amsthm}
\usepackage{amsmath}
\usepackage{amssymb}
\usepackage{fullpage}
 \usepackage{bbm}
\usepackage[dvipsnames]{xcolor}
\usepackage{amssymb}
\usepackage{hyperref}
 \def\smallsection#1{\smallskip\noindent\textbf{#1}.}
 
\def\?[#1]{\textbf{[#1]}\marginpar{\Large{\textbf{??}}}}

\def\smallsection#1{\smallskip\noindent\textbf{#1}.}
\numberwithin{equation}{section}

\title[$L^2$ linear stability of the steady state for the BGK equation on the torus]{Properties of Non-Equilibrium Steady States for the non-linear BGK equation on the torus}

\author{Josephine Evans}
\address{Warwick Mathematics Institute, Zeeman Building, University of Warwick, CV4 7AL}
\email{Josephine.Evans@warwick.ac.uk}

\author{Angeliki Menegaki}
\address{Institut des hautes études Scientifiques, 35 Rte de Chartres, 91440, Bures-sur-Yvette, France}
\email{menegaki@ihes.fr}

\keywords{Non-Equilibrium steady state, BGK model with spatial varying temperature, Existence Uniqueness and $L^2$ stability results, heat transfer}

\newtheorem{thm}{Theorem}
\newtheorem{defn}{Definition}
\newtheorem{lemma}{Lemma}
\newtheorem{prop}{Proposition}
\newtheorem{remark}{Remark}

\begin{document}
\maketitle
\begin{abstract}
\noindent We study the non-linear BGK model in 1d coupled to a spatially varying thermostat. We show existence, local uniqueness and linear stability of a steady state when the linear coupling term is large compared to the non-linear self interaction term. This model possesses a non-explicit spatially dependent non-equilibrium steady state. We are able to successfully use hypocoercivity theory in this case to prove that the linearised operator around this steady state poseses a spectral gap.
\end{abstract}
\tableofcontents

\section{Introduction}

Properties of non-equilibrium steady state (NESS) of systems in contact with thermal reservoirs (thermostats), such as existence, uniqueness, structure as well as speed of convergence towards them continue to be open major problem in statistical mechanics. 

There are a few open classical systems in a microscopical level introduced and study in order to understand -derive- the macroscopic Fourier's law of heat conduction from  microscopic Hamiltonian dynamics and for which there are some partial answers. A prototypical such example is the oscillator atom models with nearest neighbour interactions, perturbed at the boundaries with two reservoirs at different temperatures.  
Harmonic crystals is a case well studied where the non-equilibrium steady state is explicit but it corresponds to a rather unphysical scenario where Fourier's law breaks down \cite{RLL67}. 
Regarding anharmonic such crystals, even though there are partial answers regarding the non-equilibrium steady state that arises for certain cases -existence, uniqueness and exponential relaxation are ensured even quantitatively in the number of particles in certain cases- very little is known on the structure of such states. 
 Apart from microscopic oscillator chains, existence and uniqueness  is provided for a system of Newtonian particles with a long-range repulsive interaction potential \cite{GoldsteinKipnisIaniro85}. Again the situation is far from well-understood regarding a \emph{canonical} description of the state.

The same is true for systems in the mesoscopic level, i.e kinetic models described by a one-particle distribution function $f(t,x,v)$ where $t, x \in \Omega \subset \mathbb{R}^d, v \in \mathbb{R}$ are the time, space and velocity variables respectively. In particular there are no explicit solutions of the stationary (time-independent) Boltzmann equation in the more realistic non-equilibrium scenario and in many cases even the existence of such states is not known. 

In this article we study the non-linear BGK model of the Boltzmann’s equation, which is a kinetic relaxation model introduced by Bhatnagar, Gross and Krook in \cite{BGK54} as a toy model for Boltzmann flows. We continue our investigation as in \cite{EM21}, where we showed existence of a NESS for the one-dimensional model on the interval $(0,1)$ with diffusive boundary conditions playing the role of two thermal reservoirs at the boundaries.  

Here we change the boundary conditions to periodic boundaries, meaning that we consider a gas of
particles on the one-dimensional torus $\mathbb{T}$, 
 and our purpose is to study existence, uniqueness and $L_{x,v}^2$ stability for the non-linear BGK model coupled to a linear BGK operator with spatially varying temperature - being now a thermostat that acts all over the space. 

\subsection{Description of the model} We consider a gas of particles on the one-dimensional torus $\mathbb{T}$ where the collisions among the particles are represented by the nonlinear BGK operator.
\begin{equation}\label{eq:NL_BGK}
\partial_t f +v \partial_x f = \mathcal{L}f = \frac{1}{\kappa} \left( \rho_f\left(\alpha \mathcal{M}_{u_f,T_f} +(1-\alpha) \mathcal{M}_{\tau(x)}\right) - f \right), 
\end{equation} 
where $x \in \mathbb{T}, v \in \mathbb{R}$, $\tau(x)$ is a fixed function which varies with $x$ and $\alpha \in [0,1]$. Also with $\kappa$ we denote the Knudsen number, i.e.  the ratio between the mean free path and the typical observation length.

For a stationary (time-independent) solution $f(x,v)$ on the phase space  $\mathbb{T}\times \mathbb{R}$ we define the hydrodynamic moments, the spatial density $\rho_f(x)$, the bulk velocity $u_f(x)$ and the pressure $P_f(x)$ respectively as  
\begin{equation}
\begin{split}
\rho_f(x) &= \int_{-\infty}^\infty f(x,v) \mathrm{d}v, \quad \rho_f(x)u_f(x) = \int_{-\infty}^\infty v f(x,v)\mathrm{d}v  ,\quad\\ &  P_f(x) = \int_{-\infty}^\infty v^2 f(x,v) \mathrm{d}v = \rho_f(x) [T_f(x)+ u_f(x)^2 ] 
\end{split}
\end{equation}
 and then the local temperature profile corresponding to $f$ is $T_f$. 
We denote by $\mathcal{M}_{u_f, T_f}(v)$ the Maxwellian with temperature $T_f$:
\[ \mathcal{M}_{u_f,T_f}(v) = (2\pi T_f)^{-1/2} \exp \left(-\frac{|v-u_f|^2}{2T_f} \right). \] 

Our main objectives are (i) to determine the range of the parameters $\alpha, \kappa,$ and $\underline{\tau}\leq \tau \leq \bar{\tau} $ for which we have existence of a non-equilibrium stationary state $g(x,v)$ to \eqref{eq:NL_BGK}, (ii) show uniqueness of such a state for small $\alpha$'s and (iii) prove $L^2$ linear stability around $g$ via adaptations of hypocoercivity methods in this non-equilibrium scenario. 


\subsubsection{Motivation and state of the art} The main objective of our study is to understand fundamental properties -such as existence, uniqueness and possible stability- of non-equilibriums steady states in particle systems in kinetic theory. 
There are many works in this direction both for BGK and Boltzmann equations.\\
\noindent
\emph{On BGK models}:
There has been a lot of work done in this direction when one considers Boltzmann-type of models with several thermostats acting either at the boundaries or in the whole space. In particular in \cite{CELMM18, CELMM19}, the authors study the non-linear BGK model on the torus with periodic boundary conditions when scatterers at two different temperatures are imposed. In the cases covered there, one can find an explicit steady state which is spatially homogeneous, which is also unique. 

However our linear thermostat is spatially non-uniform this means our non-equilibrium steady state is non-explicit, spatially inhomogeneous and there is complex behaviour of the hydrodynamic quantities.

In \cite{EM21} we started studying as well the non-linear BGK equation on the domain $(0,1)$ when the boundary conditions are diffusive, meaning that whenever a particle hits one of the boundaries it is reflected back in the domain and its velocity is "thermalised" according to the temperature of the boundary. We study the case where the boundary temperatures are sufficiently away from each other, which is an extreme non-equilibrium regime -for a fixed Knudsen number-  and we showed existence of a steady state. A very similar scenario for the non-linear quadratic BGK model with large boundary data was studied in \cite{U92} providing also of existence, but the prescription of the boundary data is different.

\noindent
\emph{On Boltzmann models}: For the Boltzmann equation on bounded domains out-of-equilibrium there are many contributions in the perturbative case around the equilibrium, that is when the boundary temperatures are close to the equilibrium at some uniform temperature. In \cite{EGKM13} the authors constructed a unique steady state in the kinetic regime, i.e. finite Knudsen number, in the neighbourhood of the equilibrium and proved dynamical stability, generalising and expanding existence and uniqueness results in \cite{Guiraud72,Guiraud70} on convex domains.  
Existence and stability results on the NESS of Boltzmann equation expanding around small Knudsen number include \cite{A00, AEMN10, AEMN12}.
Moreover, existence of non-equilibrium steady states to the Boltzmann equation in the slab under diffuse reflection boundary conditions in a non-perturbative setting was proven in \cite{AN00}. 

Regarding spatially homogeneous Boltzmann in the presence of scatterers and of Kac's toy model -existence, uniqueness and exponential relaxation towards the NESS can be found in \cite{CLM15} and in \cite{E16}, respectively.

 In the context of deriving Fourier's law, i.e. a heat conduction law stating that the heat flux vector is proportional to the gradient of the temperature with the proportionality to be the thermal conductivity of the material, one needs to let the Knudsen number to go to $0$. Allowing thus a large number of collisions per unit time and establishing a hydrodynamic regime. 
Historically that was obtained formally by Boltzmann and Maxwell \cite{Maxwell1867,Boltzmann03} and a rigorous proof was given in \cite{ELM94,ELM95} in the slab geometry, in the close to equilibrium case. The question of deriving Fourier's law from a deterministic particle system remains a major open question in statistical mechanics \cite{BLR00, Lep16, Dhar08, BF19} and recent works include studies of harmonic as well as anharmonic oscillators in contact with thermal reservoirs at the boundaries \cite{RLL67, EPR99b, RBT02, Car07, CEHRB18, Hair09, HM09, Me20, BM20, BM23}, where existence, uniqueness of a NESS was provided for a large class of interaction potentials, exponential relaxation towards it, and in some more specific cases quantitative convergence rates as a function of the number of particles.  Fourier's law has been showed to hold also in cases of harmonic atom chains perturbed by a conservative stochastic dynamics as considered in \cite{BO05}. 
 
For a collection of the recent works on the stationary Boltzmann equation of bounded general domains on the perturbative close to equilibrium regime for both finite and close to $0$ Knudsen number, we suggest the recent survey \cite{EspositoMarra20}.

\subsection{Notation} We  write $A \lesssim B$ in order to say that $A \leq C B$ for some finite constant $C$ independent of $\alpha$. We also write $f(x) = \mathcal{O}(g(x))$ to denote that there is a constant $C>0$ such that $|f(x)| \leq C |g(x)|$ and $f(z) = o(g(z))$ for $z\to z_0$ if there is for any $\varepsilon>0$  a neighbourhood $U_\varepsilon$ of $z_0$ such that $|f(z)|\leq \varepsilon |g(z)|$.  
By $L_{x,v}^2$ we denote the space with functions that are in $L^2$  for both space and velocity variables $x,v$.
Finally we write for simplification $\mathcal{M}_{T_f}(v)$ for $\mathcal{M}_{0,T_f}(v)$.

\subsection{Plan of the paper} In the rest of this section we state the main results.  In Section \ref{sec:Existence and Uniqueness} we prove the existence and uniqueness of the steady for the full non-linear problem. This is split into subsections with the basic steps of the proof: in subsec. \ref{subs:Exist_and_uniq_linear} we prove existence and uniqueness for the linear version of the model, in subsec. \ref{subsec:bounds on pressur density temper} we provide with upper and lower bounds on the moments of the NESS and the requirements for the fixed point theorem are proved in \ref{subsec:contraction mapping}.
 In Section \ref{sec:Linear stability} we prove $L^2_{x,v}$ stability of the linearised operator around the NESS, by first proving a microscopic coercivity, i.e. in $v$ variable, subsec. \ref{subsec: micro coerc},  a macroscopic coercivity, i.e. in $x$ variable, subsec. \ref{subsec: macro coerc}, and boundedness of certain operators, subsec. \ref{subsec:boundedness auxiliar oper}, introduced to modify the entropy and to run our hypocoericivity argument.

\subsubsection{Techniques and main results}
Our first main theorem gives existence and a qualified uniqueness result for a steady state.

\begin{thm}[Existence and local uniqueness] \label{thm:existence}
Suppose there exist constants $\underline{\tau}$ and $\bar{\tau}$ such that $\underline{\tau} \leq \tau(x) \leq \bar{\tau}$ for all $x \in \mathbb{T}.$ Then if $\underline{\tau}$ is sufficiently large and $\alpha$ is sufficently small (depending on $\underline{\tau}, \bar{\tau}$), there exists a steady state $g=g(x,v)$ to the non-linear equation \eqref{eq:NL_BGK}. Moreover there exist constants $\underline{T}$ and $\bar{T}$ depending on $\alpha, \underline{\tau}, \bar{\tau}$ such that the temperature profile corresponding to this steady state $T_g(x) = \int_{-\infty}^\infty v^2 g(x,v) \mathrm{d}v / \int_{-\infty}^\infty g(x,v) \mathrm{d}v$ satisfies
\[ \underline{T} \leq T_g(x) \leq \bar{T}. \] Finally $g$ is the unique steady state in the class of functions $g$ satisfying this constraint.
\end{thm}
This theorem is proved by a contraction mapping argument on the temperature. We first show that for a specific set of parameters $\alpha,\underline{\tau}, \bar{\tau}$, upper and lower bounds on the temperature, pressure and spatial density are preserved. Then in this set of parameters we show that the mapping on the temperature is in fact contractive. To construct the map we freeze the non-linearity in the equation to produce a linear equation which can be related to a Markov process. We the use this to show that linear equation has a unique steady state for which can define its temperature. This gives us a map from one temperature to another. Our estimates showing bounds on the temperatures rely on a representation of the steady state of this linear equation found through Duhamel's formula. This is an adaption of our techniques from \cite{EM21} where we performed a Schauder fixed point argument on a similar equation. 

\begin{thm}[Linear stability of the steady state] \label{thm:stability}
For $\underline{\tau}$ sufficiently large that the result of Theorem \ref{thm:existence} applies and $\alpha$ sufficiently small the linear equation found by linearizing \eqref{eq:NL_BGK} around $g$ has a spectral gap in $L^2(g^{-1})$.
\end{thm}
This theorem is proved by adapting the $L^2$ hypocoercivity theorem of Dolbeault-Mouhot-Schmeiser \cite{DMS15}. The novelty here is that we push forward these techniques to an out-of-equilibrium system, meaning that we deal with a non-explicit, spatially inhomogeneous, steady state. Therefore we work on a functional space, $L^2(g^{-1})$, with non-explicit weights. The $L^2$ hypocoercivity theory relies on verifying assumptions. In our case the most challenging is the \emph{microscopic coercivity} assumption. Difficulties in verifying this assumption are related to the fact that the non-equilibrium nature of our dynamics mean there are complex behaviours relating to flow of macroscopic quantities which appear as potentially growing modes in the space spanned by $g, vg, v^2g$. We are able to control these terms when $\alpha$ is small. Doing this involves bounding the normalised third and fourth moments (skewness and kurtosis) of the steady state $g$ that we find.

\begin{remark} 
Both Theorems \ref{thm:existence} and \ref{thm:stability} require a finite Knudsen number $\kappa$ of order $\mathcal{O}(1)$, meaning that we stay on the kinetic - mesoscopic level, without passing to the hydrodynamic regime. We remark that this was the case also for the steady state we found in \cite{EM21} when the model is subject to diffusive boundary conditions at two different -sufficiently large- temperatures far away from each other. 
\end{remark}
\begin{remark}
Regarding the stability and uniqueness of the steady state that is provided here for small $\alpha$'s, let us remark that it is due to the thermostat at temperature $\tau(x)$ acting all over the space. This term is helping the system to be stabilised as it imposes constrains everywhere in the space. That comes in contrast with the case when the thermostats act just on the boundaries, as in \cite{EM21}. This means that the steady state found for the system studied in \cite{EM21} when the difference of the boundary temperatures is large, could be possibly unstable. 
\end{remark}

\subsubsection{Discussion and open problems}

There are several natural next steps from this work. For the model studied in the paper two natural next steps would be:
\begin{itemize}
\item Studying the non-linear stability of this equation by first utilizing the toolbox in \cite{GMM17} to prove weighted $L^\infty$ stability for the linearized equation;
\item Looking at the case of small Knudsen number ($\kappa$) where we expect that one should be able to use the fact that the steady state will be close to Maxwellian.
\end{itemize}
We comment that we could also look at the situation where $\tau$ does not vary very much. That is to say $|\bar{\tau}- \underline{\tau}|$ is small and we can weight by the Maxwellian at a fixed temperature. In fact we expect this case to be rather straightforward. We have included a brief proof of linear stability in the appendix \ref{appendix} under the assumption that the temperature of the steady state is close to a uniform temperature, close to $\bar{\tau}, \underline{\tau}$ in this case. We expect that this assumption could be verified by arguments similar to our contraction mapping argument to show that the temperature of the steady state is a Lipschitz function of $\tau(x)$ when $\tau$ is considered as a function in $L^\infty$. These would be more involved than the linear stability.

We would like to be able to perform a linear stability analysis similar to the one in this paper for the equation studied in \cite{EM21} with diffusive boundary conditions. For this equation it is not at all clear that the steady state found is stable when the difference between the boundary temperatures is large. This comment is also made in the related paper of Ukai \cite{U92}. Briefly, if we try and run an argument similar to the linear stability argument in this paper we are not able to verify the \emph{microscopic coercivity assumption} and we can in fact show that it is possible for the $L^2(g^{-1})$ norm to grow for certain initial data. This is caused by complex behaviour of the hydrodynamic quantities and these terms cannot be controlled by terms coming from the thermostat since it only acts on the boundary. 

\section{Existence and uniqueness}\label{sec:Existence and Uniqueness}
We consider the linear equation 
\begin{equation}\label{eq:LBGK}
 \partial_t f +v \partial_x f = Lf= \frac{1}{\kappa} \left( \rho_f\left(\alpha \mathcal{M}_{T(x)}(v) +(1-\alpha) \mathcal{M}_{\tau(x)}(v)\right) - f \right)
\end{equation}
and we write $g=g^T(x,v)$ to be the steady state of \eqref{eq:LBGK} corresponding to the temperature profile $T \in C(\mathbb{T})$. First we prove the existence and uniqueness of such a state through Doeblin's theorem.

\subsection{Doeblin's argument for existence and uniqueness of a steady state to the linear equation} \label{subs:Exist_and_uniq_linear}

We need to show the existence and uniqueness of a steady state for the linear PDE  \eqref{eq:LBGK}
  We first construct a stochastic process the law of which is a weak solution of the PDE \eqref{eq:LBGK}. Then using Doeblin's theorem from Markov process theory we show existence and uniqueness of the steady state.

To construct the stochastic process first let us generate a Poisson proces s with rate $\frac{1}{\kappa}$ and call $J_1, J_2, J_3, \dots$ the jump times of the Poisson process. In order to construct the stochastic process $X_t, V_t$ we proceed iteratively: Suppose we have it up to time $J_i$, then  for 
$$J_i < t \leq J_{i+1},\ \text{ we set } X_t = X_{J_i}+(t-J_i)V_{J_i}$$ and  
$$
J_i < t < J_{i+1},\ V_t = V_{J_i}, 
$$
while for $t= J_{i+1}$ we generate a new velocity independent of everything except $X_{J_{i+1}}$. It is straightforward to check that this gives the right equation: By taking a test function $\phi$ on $C_c^{\infty}((0,\infty)\times \mathbb{T} \times \mathbb{R})$, Taylor expanding the quantity $s^{-1}\mathbb{E}\left( \phi(t+s,X_{t+s}, V_{t+s}) - \phi(t,X_{t}, V_{t}) \right)$ and taking the limit as $s \to 0$ we recover indeed a weak solution of our PDE \eqref{eq:LBGK}.

Before applying Doeblin's theorem, let us first give a general statement of it. 
\begin{thm}[Doeblin's theorem]
Let $Z_t$ be a continuous time Markov process on a state space $\mathcal{Z}$. Let us write $f_t^{z_0}$ for the law of $Z_t$ conditional on $Z_0  = z_0$. Suppose that there exists a time $t_*>0$ a constant $\beta \in (0,1)$ and a probability measure $\nu$ on $\mathcal{Z}$ so that for any $z_0 \in \mathcal{Z}$ we have 
\[ f_{t_*}^{z_0} \geq \beta \nu, \] then the Markov process has a unique steady state.
\end{thm}

In order to prove the conditions of Doeblin's theorem for our particular Markov process, the most convenient way is to look at the PDE form of the equation.
\begin{prop}
For the linear equation with $f_0 = \delta_{x_0,v_0}$ we have $\beta \in (0,1)$ and $\nu$ a probability density such that \[ f_{t} \geq \beta \nu.\]
\end{prop}
\begin{proof}
First since $\tau, T$ are bounded above and below we can find $c>0$ so that
\[ \alpha \mathcal{M}_{T(x)}(v) + (1-\alpha) \mathcal{M}_{\tau(x)}(v) \geq c1_{|v| \leq 1} \]  uniformly in $x$. Then Duhamel's formula yields the lower bound
\begin{equation} 
\begin{split}
 f(t,x,v) \geq e^{-t/\kappa} f(0, x-vt, v) +& \int_0^t e^{-(t-s)/\kappa} \int f(s, x-v(t-s), u) \mathrm{d}u\times \\  &\times \left( \alpha \mathcal{M}_{T(x-v(t-s))}(v) + (1-\alpha) \mathcal{M}_{\tau(x-v(t-s))}(v) \right)\mathrm{d}s. 
 \end{split} 
 \end{equation}
Then using the notation $(S_tf)(x,v) = f(x-tv, v)$, i.e. the semigroup of the transport part, and 
$P_c f := c 1_{|v| \leq 1}\int_{\mathbb{R}} f(x,u)\mathrm{d}u $ we write
\[ f(t,x,v) \geq e^{-t/\kappa}S_t f_0 + \int_0^t e^{-(t-s)/\kappa} S_{t-s}P_c f_s \mathrm{d}s. \] 
This implies that 
$f(t,x,v) \geq e^{-t/\kappa} S_t f_0$. Now by substituting this into the second term of the equation, we get
\[ f(t,x,v)\geq \int_0^t e^{-t/\kappa} S_{t-s}P_c S_s f_0 \mathrm{d}s, \] 
and by substituting the new lower bound in again we get
\[ f(t,x,v) \geq \int_0^t \int_0^s e^{-t/\kappa} S_{t-s}P_c S_{s-r} P_c S_r f_0 \mathrm{d}r \mathrm{d}s. \]
Now by our initial assumption $f_0 = \delta_{(x_0,v_0)}$, we have that
\[ S_r \delta_{x_0,v_0} = \delta_{x_0+rv_0, v_0}. \] 
Therefore we have
\begin{equation}
\begin{split} 
 P_c S_r \delta_{x_0,v_0} &= c\delta_{x_0+rv_0}1_{|v| \leq 1},\\ 
S_{s-r}P_c S_r \delta_{(x_0,v_0)} &= c \delta_{x_0+rv_0}(x-(s-r)v) 1_{|v| \leq 1}, \text{ and }\\
P_c S_{s-r}P_c S_r \delta_{(x_0,v_0)} &=\frac{c^2}{(s-r)^d} 1_{|v| \leq 1} 1_{|x_0+rv_0 - x| \leq (s-r)1}. 
\end{split}
\end{equation} 
Since $x$ is on the torus as long as $(s-r)1 \geq \sqrt{d}=2$, where $d$ is the dimension, then the last indicator function is always satisfied. So
\[  P_c S_{s-r}P_c S_r \delta_{(x_0,v_0)} \geq \frac{c^2}{(s-r)^d} 1_{|v| \leq 1} \] and hence
\[ S_{t-s} P_c S_{s-r} P_c S_r \delta_{(x_0,v_0)} \geq \frac{c^2}{(s-r)^d} 1_{|v| \leq 1}\] whenever $(s-r)1 \geq \sqrt{d}$. Finally choose $t_* = 2 \sqrt{d}=2$, in our one-dimensional setting, and integrate over $0<r<s<t_*$ so that indeed we have  $s-r \geq \sqrt{d}$,  to get our lower bound as needed.
\end{proof}

  \subsection{Representation formula for the steady state and moment bounds}\label{subsec:bounds on pressur density temper}
  In this subsection we proceed by getting a handy representation of our steady state for the linear BGK \ref{eq:LBGK} through an application of Duhamel's formula. This will allow us to estimate then by below and above the density and the pressure in the steady state. 
  
\begin{lemma}  Let $g^T$ be a solution to \eqref{eq:LBGK}.
We can get the following representation formula  
\[ g^T(x,v) = \int_0^1 \frac{e^{-(1-y)/\kappa |v|}}{\kappa |v| (1-e^{-1/\kappa |v|})}  \rho_g(x + \operatorname{sgn}(v)y) \left( \alpha\mathcal{M}_{T(x + \operatorname{sgn}(v)y)}(v) + (1-\alpha) \mathcal{M}_{\tau(x+ \operatorname{sgn}(v) y)}(v) \right) \mathrm{d}y 
\]
\end{lemma}
\begin{proof} Let $v>0$ and notice that Duhamel's formula yields
\begin{align*} 
e^{t/\kappa} g(x+vt,v) - g(x,v) = \int_0^t \frac{e^{s/\kappa}}{\kappa} \rho_g(x+vs)  
\left( \alpha
\mathcal{M}_{T(x+vs)}(v) + (1-\alpha) \mathcal{M}_{\tau(x+vs)}(v) \right)
 \mathrm{d}s.
\end{align*}
Since $x \in \mathbb{T}$, by inserting $t=\frac{1}{|v|}$ and $y=vs$, we write 
\begin{align*} 
(e^{1/\kappa |v|}-1 ) g(x,v) & = \int_0^1 \frac{e^{y/\kappa |v|} }{\kappa |v|}\rho_g (x+y)  \left( \alpha
\mathcal{M}_{T(x+y)}(v) + (1-\alpha) \mathcal{M}_{\tau(x+y)}(v) \right)   \mathrm{d} y, 
\end{align*} or equivalently 
\begin{align*} 
 g(x,v)  &=   \int_0^1 \frac{ e^{ -(1-y)/ \kappa |v|} }{ \kappa |v| (1-e^{-1/\kappa |v|}) }\rho_g (x+y)  \left( \alpha
\mathcal{M}_{T(x+y)}(v) + (1-\alpha) \mathcal{M}_{\tau(x+y)}(v) \right)    \mathrm{d}y.
\end{align*} 
While same calculations for $v<0$: 
\begin{align*} 
 g(x,v)  &=   \int_0^1 \frac{ e^{ -(1-y)/ \kappa |v|} }{ \kappa |v| (1-e^{-1/\kappa |v|}) }\rho_g (x-y)  \left( \alpha
\mathcal{M}_{T(x-y)}(v) + (1-\alpha) \mathcal{M}_{\tau(x-y)}(v) \right)    \mathrm{d}y.
\end{align*} 
This implies the stated representation formula for the steady state.
\end{proof}

We proceed by providing explicit upper and lower bounds on the density $\rho_g$, the pressure $P_g$ and consequently the temperature $T_g$. 
Our ultimate goal is, through our bounds on $T_g$, to determine the values of $\alpha, \bar{\tau}, \underline{\tau}$ for which (i) the non-equilibrium steady state $g$ of the linear model coincides with a steady state to the non-linear model, providing with existence, and (ii) this steady state is unique.

\begin{lemma}  \label{lem: bounds_density}
Assuming that  $\underline{T}\leq \underline{\tau}$,  for $\underline{\tau} \geq 1$ and $\alpha \leq \frac{1-\underline{\tau}^{-1/2}}{\underline{T}^{-1/2}-\underline{\tau}^{-1/2}}$, we have uniformly in $x \in \mathbb{T}$, 
$$ 1 - \frac{1}{\kappa}\Big[ \alpha \underline{T}^{-1/2}  + (1-\alpha) \underline{\tau}^{-1/2} \Big] \lesssim \rho_g(x) \lesssim 
4  \left[ \frac{\alpha}{\kappa \underline{T}^{1/4}} + \frac{1-\alpha}{\kappa \underline{\tau}^{1/4}} \right]^2. $$
\end{lemma} 


\begin{proof} 
\textit{Upper bound on the density}: 
Looking then at the asymptotics of the integrand, we first notice that 
$$ \frac{e^{-(1-y)/\kappa |v|}}{\kappa |v| (1-e^{-1/\kappa|v|}) } \leq \frac{1}{\kappa} \max\left\{  1+ \frac{e^{-1}}{1-y}, \frac{1}{|v|(1-e^{-1/\kappa |v|})}  \right\}$$ where the first term accounts for the big velocities and the second for the small. 

Then 
\begin{align*} 
&\int_0^{\infty} g^T(x,v)  \mathrm{d}v = \int_0^{\infty} \int_0^1 \frac{e^{-(1-y)/\kappa |v|}}{\kappa |v| (1-e^{-1/\kappa|v|}) } \rho_g(x+y) \left( \alpha
\mathcal{M}_{T(y+x)} + (1-\alpha) \mathcal{M}_{\tau(y+x)} \right) \mathrm{d} y \mathrm{d} v\\ & = 
\int_0^1 [  \alpha \mathcal{K}_1(T,y)+ (1-\alpha) \mathcal{K}_1(\tau,y) ] \rho_g (x+y) \mathrm{d} y.
\end{align*} 
where 
\begin{equation} \label{eq:def of F}
\mathcal{K}_1(T(y+x), y) = \int_0^\infty  \frac{1}{\kappa |v|} e^{-(1-y)/\kappa |v|}(1-e^{-1/\kappa|v|})^{-1} \mathcal{M}_{T(y+x)}(v) \mathrm{d}v. \end{equation} 

Now in order to bound $\mathcal{K}_1$, we denote by $v_*$ the value so that  
\begin{align*} 
\mathcal{K}_1(T, y)& \leq \int_{v_*}^{\infty}
\frac{1}{v_* \kappa} \left( 1+ \frac{e^{-1}}{1-y}\right)  \mathcal{M}_{T(x+y)}(v)  \mathrm{d}v  + \int_0^{v_*}   \frac{ e^{-(1-y)/\kappa |v|}}{\kappa |v|}(1+e^{-1/\kappa |v|})   \mathcal{M}_{T(x+y)}(v)    \mathrm{d}v \\ 
&\lesssim
 \mathcal{O}\left(  \frac{1+\frac{e^{-1}}{1-y} }{\kappa v_*} +  \frac{v_*}{\kappa T(x+y)^{1/2}} \right).
\end{align*} 
Optimising over $v_*$ we get that for $v_*= T(y+x)^{1/4} \left(\frac{1}{2}  +  \frac{e^{-1}}{2(1-y)} \right)^{1/2}$, 
\begin{align} \label{eq:estimat on F}
\mathcal{K}_1(T,y+x) \lesssim \frac{1}{\kappa} \left( \frac{1}{2}+ \frac{e^{-1}}{2(1-y)} \right)^{1/2} T(y+x)^{-1/4}. 
\end{align} 
This leads to the following upper bound on the density for positive $v$ for some $\delta \in (0,1)$: 
\begin{equation} \label{eq:estim_density_v>0}
\begin{split}
 &\int_0^{\infty} g^T(x,v) \mathrm{d}v \lesssim 
 \int_0^{1} \Bigg[ \frac{\alpha}{\kappa}  \left( 1+ \frac{e^{-1}}{1-y} \right)^{1/2} T(y+x)^{-1/4} 
 \\ &\qquad\qquad\qquad\qquad\qquad\qquad\qquad + \frac{1-\alpha}{\kappa}   \left( 1+ \frac{e^{-1}}{1-y} \right)^{1/2} \tau(y+x)^{-1/4} \Bigg] \rho_g (x+y)  \mathrm{d}y\\ 
 & \lesssim  \frac{\left(1+\frac{1}{e \delta} \right)^{1/2}}{\kappa} \int_0^{1-\delta} \rho_g(y+x) \left[\alpha T(x+y)^{1/4} + (1-\alpha) \tau(x+y)^{1/4}  \right]  \mathrm{d}y  +\\ &
 \int_{1-\delta}^1 \left[  \frac{\alpha}{\kappa} \left( \frac{1}{2}+ \frac{e^{-1}}{2(1-y)} \right)^{1/2} T(y+x)^{-1/4} +   \frac{1-\alpha}{\kappa} \left( \frac{1}{2}+ \frac{e^{-1}}{2(1-y)} \right)^{1/2} T(y+x)^{-1/4} \right] \rho_g (x+y)  \mathrm{d}y\\ 
 &\lesssim \left[ \frac{\alpha}{\kappa \inf_y T(x+y)^{1/4}} + \frac{1-\alpha}{\kappa \inf_y \tau(x+y)^{1/4}} \right] 
 \bigg[ \left(1+\frac{1}{e \delta} \right)^{1/2} \int_{0}^{1-\delta} \rho_g (x+y)  \mathrm{d}y 
 \\ &\qquad\qquad\qquad\qquad\qquad\qquad\qquad\qquad\qquad
 + \| \rho_g \|_{\infty} \int_{1-\delta}^1  \left( \frac{1}{2}+ \frac{e^{-1}}{2(1-y)} \right)^{1/2}   \mathrm{d}y \bigg]. 
 \end{split}
\end{equation}
Expanding for $\delta$ small we notice that  $\int_{1-\delta}^1  \left( \frac{1}{2}+ \frac{e^{-1}}{2(1-y)} \right)^{1/2}   \mathrm{d} y  = \mathcal{O}\left( i \delta \operatorname{cot}(1)\operatorname{csc}(1)/2   + \sqrt{\delta (1+\delta)} \right) = \mathcal{O}\left( \sqrt{\delta} \right)$. Taking into account also the negative velocities we eventually bound the last line in \eqref{eq:estim_density_v>0} as 
\begin{equation}
\begin{split}
\rho_g(x)& \lesssim  \left[ \frac{\alpha}{\kappa \underline{T}^{1/4}} + \frac{1-\alpha}{\kappa \underline{\tau}^{1/4}} \right] \left[ C_1 \left(1+\frac{1}{e \delta} \right)^{1/2} + C_2 \| \rho_g \|_{\infty} \sqrt{\delta} \right] \\ 
& \lesssim  \left[ \frac{\alpha}{\kappa \underline{T}^{1/4}} + \frac{1-\alpha}{\kappa \underline{\tau}^{1/4}} \right] \left[ \frac{C_1}{\sqrt{\delta}} +  \| \rho_g \|_{\infty} C_2 \sqrt{\delta} \right]
 \end{split}
\end{equation}
for some finite constants $C_1,C_2$. Take now $\sqrt{\delta} = \frac{1}{2C_2 \left[ \frac{\alpha}{\kappa \underline{T}^{1/4}} + \frac{1-\alpha}{\kappa \underline{\tau}^{1/4}} \right] }$ to find that for all $x \in \mathbb{T}$
\begin{equation}
 \rho_g (x) \lesssim 4C_1C_2  \left[ \frac{\alpha}{\kappa \underline{T}^{1/4}} + \frac{1-\alpha}{\kappa \underline{\tau}^{1/4}} \right]^2. 
\end{equation}

\textit{Lower bound on the density}: 
Let $\Lambda>0$ sufficiently large so that $\frac{1}{ (e^{1/\kappa|v|}-1)} \sim \kappa |v|$ for $|v| > \Lambda$ and write 
\begin{equation}
\begin{split}
& \int_0^{\infty} g^T(x,v)  \mathrm{d}v  = \int_0^{\infty} \int_0^1 
 \frac{e^{-(1-y)/\kappa |v|}}{\kappa |v| (1-e^{-1/\kappa |v|})}  \rho_g(x+y) \left[ \alpha\mathcal{M}_{T(x +y)}(v) + (1-\alpha) \mathcal{M}_{\tau(x+y)}(v) \right] \mathrm{d}y \mathrm{d}v \\ &  
 \gtrsim \int_0^1 \int_{\Lambda}^{\infty} \frac{e^{-(1-y)/\kappa |v|}}{\kappa |v| (1-e^{-1/\kappa |v|})}  \rho_g(x+y) \left[ \alpha\mathcal{M}_{T(x +y)}(v) + (1-\alpha) \mathcal{M}_{\tau(x+y)}(v) \right] \mathrm{d}v \mathrm{d}y \\
 & \gtrsim  \int_0^1 \rho_g(x+y)  \int_{\Lambda}^{\infty}    
 \frac{1}{ (e^{1/\kappa|v|}-1)}
  \frac{1}{\kappa |v|} 
   \left[ \alpha\mathcal{M}_{T(x +y)}(v) + (1-\alpha) \mathcal{M}_{\tau(x+y)}(v) \right] \mathrm{d}v \mathrm{d}y \\ 
   & \gtrsim  \int_0^1 \rho_g(x+y)   \left[  \int_{\frac{\Lambda}{\sqrt{T(x+y)}} }^{\infty} \frac{\alpha}{\kappa \sqrt{T(x+y)}} \mathcal{M}_1(v) \mathrm{d}v +  \int_{\frac{\Lambda}{\sqrt{\tau(x+y)}} }^{\infty} \frac{1-\alpha}{\kappa \sqrt{\tau(x+y)}} \mathcal{M}_1(v) \mathrm{d}v \right] \mathrm{d}y \\ & 
   \gtrsim \sqrt{\frac{\pi}{2}}  - \sqrt{\frac{\pi}{2}}  \Big( \alpha T(x+y)^{-1/2} + (1-\alpha) \tau(x+y)^{-1/2} \Big)\\ 
   & \gtrsim 1 - \Big[ \alpha \underline{T}^{-1/2}  + (1-\alpha) \underline{\tau}^{-1/2} \Big].
 \end{split} 
\end{equation}
This implies the stated lower bound on the density which is a non-negative quantity under the constraints in the hypothesis. 
\end{proof} 

Now we proceed to get upper and lower bounds on the pressure $P_{g^T}(x)$. First we notice that 
$u_{g^T}(x) = \frac{\int v g^T(x,v)  \mathrm{d}v}{\int g^T(x,v)  \mathrm{d}v} =0$. This is since due to mass conservation $\partial_x ( u_{g^T}(x)\rho_{g^T}(x) )=0$, and so $\int vg^T(x,v)  \mathrm{d}v = \int v g^T(0,v)  \mathrm{d} v =0$. Moreover $\partial_x P_{g^T}(x) = -\frac{1}{\kappa} u_{g^T}(x)\rho_{g^T}(x) =0, $ so that the pressure is constant in $x$.

\begin{lemma}  \label{lem: bounds_press} 
For $\alpha$ sufficiently small so that the following holds true
\begin{align} \label{eq:constr_LB on P posit}
 \alpha \left(\sqrt{\underline{\tau}} - \sqrt{\underline{T}}  + \frac{4\pi}{\kappa^2 \underline{T}^{1/4}\underline{\tau}^{1/4}  }
 \left(\sqrt{\underline{\tau}} -  \underline{T}^{1/4}\underline{\tau}^{1/4}  \right) 
 + \frac{2\pi}{\kappa^2}\alpha \left(  \frac{1}{\underline{T}^{1/4}} - \frac{1}{\underline{\tau}^{1/4}} \right)^2  \right) \leq \sqrt{\underline{\tau}} - \frac{2\pi}{\kappa^2} \frac{1}{\sqrt{\underline{\tau}}},\end{align} 
we have uniformly in $x \in \mathbb{T}$, 
\begin{align*} 
 \left[ \alpha \frac{\underline{T}^{1/2}}{2\pi} + (1-\alpha)\frac{\underline{\tau}^{1/2}}{2\pi}  \right]  
-& \frac{1}{\kappa}  \left[ \frac{\alpha}{\kappa \underline{T}^{1/4}} + \frac{1-\alpha}{\kappa \underline{\tau}^{1/4}} \right]^2 \lesssim 
P_{g^T}(x) \\ & 
 \lesssim  C_0 \left[ \frac{\alpha}{\kappa \underline{T}^{1/4}} + \frac{1-\alpha}{\kappa \underline{\tau}^{1/4}} \right]^2 \left[ \alpha \sqrt{\overline{T}}+ (1-\alpha)  \sqrt{\overline{\tau}} \right]
 \end{align*} 
for some finite temperature-independent constant $C_0$. 
\end{lemma}
\begin{proof}
\textit{Upper Bound on the pressure $P_{g^T}$}: Using the representation formula for the stationary state $g^T$, we need to control the following for positive velocities
\begin{align*} 
J_+ (x) :=\int_0^1 \rho_{g^T}(x+y) \int_0^{\infty} |v|  \frac{e^{-(1-y)/\kappa |v|}}{\kappa (1-e^{-1/\kappa |v|})}  \left( \alpha
\mathcal{M}_{T(x+y)}(v) + (1-\alpha) \mathcal{M}_{\tau(x+y)}(v) \right)  \mathrm{d} v.
\end{align*} 
 Let $v_*$ positive and finite to be chosen later. Using that for large velocities $\frac{e^{-(1-y)/\kappa |v|} }{1-e^{-1/\kappa |v|}} \leq 1+\frac{1}{e(1-y)}$ while for bounded velocities $\frac{e^{-(1-y)/\kappa |v|} }{1-e^{-1/\kappa |v|}} $ is bounded uniformly in $y$, we estimate 
 \begin{align*} 
  \int_{0}^{\infty} \frac{v e^{y/\kappa |v|}}{e^{1/\kappa |v|}-1}\mathcal{M}_{T(x+y)}(v)   \mathrm{d}v  &= 
 \int_{v_*}^{\infty} \frac{v e^{y/\kappa |v|}}{e^{1/\kappa |v|}-1}\mathcal{M}_{T(x+y)}(v)   \mathrm{d}v  +
   \int_0^{v_*} \frac{v e^{y/\kappa |v|}}{e^{1/\kappa |v|}-1}\mathcal{M}_{T(x+y)}(v)   \mathrm{d}v
 \\  &\lesssim 
 \frac{1}{\kappa} \int_{v_*}^{\infty} \left( 1+\frac{1}{e(1-y)}\right) |v| \mathcal{M}_{T(x+y)}(v)   \mathrm{d}v  + \frac{1}{\kappa} \int_0^{v_*} v \mathcal{M}_{T(x+y)}(v)   \mathrm{d}v  \\
  &\lesssim  \frac{ \left( 1+\frac{1}{e(1-y)}\right)}{\kappa} \frac{\sqrt{T(x+y)}}{v_*} + \frac{v_* \sqrt{T(x+y)}}{\kappa}.
\end{align*} 
After optimizing in $v_*$, i.e. taking $v_* =\left( 1+\frac{1}{e(1-y)}\right)^{1/2}$, we find 
\begin{align} \label{eq:upper bound on G}
 \int_{0}^{\infty} \frac{v e^{y/\kappa |v|}}{e^{1/\kappa |v|}-1}\mathcal{M}_{T(x+y)}(v)   \mathrm{d}v \lesssim \frac{\sqrt{\overline{T}}}{\kappa} \left( 1+\frac{1}{e(1-y)}\right)^{1/2}.
\end{align} 
Finally 
\begin{align*} 
J_+(x)  \lesssim \int_0^1 \rho_{g^T}(x+y) \sqrt{\left( 1+\frac{1}{e(1-y)}\right)} \left[ \frac{\alpha \sqrt{\overline{T}}}{\kappa}+ \frac{(1-\alpha)  \sqrt{\overline{\tau}} }{\kappa} \right]  \mathrm{d} y
\end{align*} 
which employing our upper bound on the density $\rho_{g^T}$, we have 
\begin{align}
J_+(x) \lesssim 
4  \left[ \frac{\alpha}{\kappa \underline{T}^{1/4}} + \frac{1-\alpha}{\kappa \underline{\tau}^{1/4}} \right]^2 \left[ \alpha \sqrt{\overline{T}}+ (1-\alpha)  \sqrt{\overline{\tau}} \right]    \int_0^1\sqrt{\left( 1+\frac{1}{e(1-y)}\right)}  \mathrm{d}y. 
\end{align}
We apply the analogous arguments same for the negative velocities to get that
\begin{align}
P_{g^T} (x) \lesssim  C_0\left[ \frac{\alpha}{\kappa \underline{T}^{1/4}} + \frac{1-\alpha}{\kappa \underline{\tau}^{1/4}} \right]^2 \left[ \alpha \sqrt{\overline{T}}+ (1-\alpha)  \sqrt{\overline{\tau}} \right]
\end{align}
for a finite constant $C_0$.
\\

\noindent
\textit{Lower Bound on the pressure $P_{g^T}$}: For this we bound $e^{-(1-y)/\kappa |v|} \geq 1- \frac{1-y}{\kappa |v|}$, so that for  positive velocities
\begin{equation} 
\begin{split} 
\int_0^{\infty} |v| \frac{e^{-(1-y)/\kappa |v|}}{1-e^{-1/\kappa |v|}} \mathcal{M}_{T(x+y)}(v)  \mathrm{d}v & \gtrsim \int_0^{\infty} v\mathcal{M}_{T(x+y)}(v) - \int_0^{\infty} |v|  \frac{1-y}{\kappa |v|}  \mathcal{M}_{T(x+y)}(v) \mathrm{d}v 
\\ & = \frac{\sqrt{T(x+y)}}{\sqrt{2\pi}} - \frac{1-y}{2\kappa} \gtrsim \frac{\sqrt{\underline{T}}}{2\pi}  - \frac{1-y}{2\kappa}.
\end{split} 
\end{equation}  
Consequently 
\begin{align*} 
 J_+(x)& \gtrsim \int_0^1 \rho_{g^T}(x+y) \left[ \alpha \frac{\underline{T}^{1/2}}{2\pi} + (1-\alpha)\frac{\underline{\tau}^{1/2}}{2\pi}  \right]  \mathrm{d} y  - \int_0^1 \rho_{g^T}(x+y) \frac{1-y}{2\kappa} \mathrm{d}y  \\ 
& \gtrsim  \left[ \alpha \frac{\underline{T}^{1/2}}{2\pi} + (1-\alpha)\frac{\underline{\tau}^{1/2}}{2\pi}  \right]  
- \frac{1}{\kappa}  \left[ \frac{\alpha}{\kappa \underline{T}^{1/4}} + \frac{1-\alpha}{\kappa \underline{\tau}^{1/4}} \right]^2.
\end{align*}
Collecting the same estimates for the negative $v$'s we write 
$$ P_{g^T}(x) \gtrsim \left[ \alpha \frac{\underline{T}^{1/2}}{2\pi} + (1-\alpha)\frac{\underline{\tau}^{1/2}}{2\pi}  \right]  
- \frac{1}{\kappa}  \left[ \frac{\alpha}{\kappa \underline{T}^{1/4}} + \frac{1-\alpha}{\kappa \underline{\tau}^{1/4}} \right]^2.
$$
The constraint \eqref{eq:constr_LB on P posit} is required to make this lower bound non-negative. 
\end{proof}

\subsection{Contraction of the mapping $\mathcal{F}$} \label{subsec:contraction mapping}
 Since we have ensured the uniqueness of the steady state $g = g^T$ for the equation \eqref{eq:LBGK}, as proved in subsection \ref{subs:Exist_and_uniq_linear}, we are allowed to define the mapping 
 $$\mathcal{F}: C(\mathbb{T}) \to C(\mathbb{T}),\qquad  (\mathcal{F} T) (x)=  \frac{\int_{\mathbb{T}} |v|^2 g(x,v)  \mathrm{d} v}{\int_{\mathbb{T}} g(x,v)  \mathrm{d} v} = \frac{P_{g^T}}{\rho_{g^T}}(x).$$
 
Our goal is to show that this mapping is contractive, which implies that there is a fixed point. Existence of such fixed point, i.e. $T \in C(\mathbb{T})$ so that $T(x) = T_g(x) $ for all $x \in \mathbb{T}$, implies that the steady state found for the linear model \eqref{eq:LBGK} corresponds to a steady state in the non-linear BGK model \eqref{eq:NL_BGK}. 

Before proceeding with the statement let us define the following set that determines the Lipschitz constants for $\mathcal{F}$, in terms of $\alpha$ :
 
\begin{align*}
\mathcal{C}_{ \underline{T}, \underline{\tau}, \kappa} &= (0,1)\cap\Bigg\{ \alpha:  8C \alpha \left( 1 - \Big[ \alpha \underline{T}^{-1/2}  + (1-\alpha) \underline{\tau}^{-1/2} \Big] \right)^2
\left[ \frac{\alpha}{ \underline{T}^{1/4}} + \frac{1-\alpha}{ \underline{\tau}^{1/4}} \right]^4\times 
\\ &  \qquad \qquad\qquad\qquad
 \times \Bigg[
   \frac{ \left( (1-\alpha)  \sqrt{\overline{\tau}} + \alpha \sqrt{\overline{T}} \right)
 }{ \sqrt{\kappa}\underline{T}^2 \Big( 1- 2 [(1-\alpha) \underline{\tau}^{-1/4} + \alpha\underline{T}^{-1/4} ] \Big) }\left( 1+ \frac{4}{\kappa} \right) + \frac{1}{\kappa \underline{T}^{1/8}}
   \Bigg]  \in [0, 1)
\Bigg\}.
\end{align*}

 \begin{prop} 
 Let two temperature profiles $T_1, T_2 \in C(\mathbb{T})$. Assuming that 
  $$ \inf_{x\in \mathbb{T}} \tau(x)=: \underline{\tau} > \inf_{x\in \mathbb{T}} T(x):=\underline{T},\ \underline{\tau}> 2^4 \text{ and }  \alpha \in \mathcal{C}_{ \underline{T}, \underline{\tau}, \kappa},$$ the mapping $\mathcal{F}$ exhibits a contraction property
 $$ \| \mathcal{F} (T_1) -\mathcal{F} (T_2) \|_{L^{\infty}(\mathbb{T})} \leq  C_{\alpha, \underline{T}, \underline{\tau}, \kappa} \| T_1-T_2\|_{L^{\infty}(\mathbb{T})}$$
 for some explicit finite constant $C_{\alpha, \underline{T}, \underline{\tau}, \kappa}  \in [0, 1)$ depending on $\alpha, \underline{T}, \underline{\tau}, \kappa$.
 \end{prop}
 \begin{proof}
We fix two functions $T_1, T_2 \in C(\mathbb{T})$ and using the upper and lower bounds on the density and the upper bound on the pressure, we calculate: 
 
 \begin{equation}\label{eq1:contract_F}
 \begin{split} 
& \vert \mathcal{F} (T_1)(x) -\mathcal{F} (T_2)(x) \vert \leq \\ &\left( 1 - \Big[ \alpha \underline{T}^{-1/2}  + (1-\alpha) \underline{\tau}^{-1/2} \Big] \right)^{-2}  
 \bigg\{
 \left[ \frac{\alpha}{\kappa \underline{T}^{1/4}} + \frac{1-\alpha}{\kappa \underline{\tau}^{1/4}} \right]^2 \left[ \alpha \sqrt{\overline{T}}+ (1-\alpha)  \sqrt{\overline{\tau}} \right]  \vert \rho_{g^{T_1}} - \rho_{g^{T_2}} \vert (x) +  \\ & 
 \qquad\qquad\qquad \qquad\qquad\qquad \qquad\qquad\qquad
  + 4  \left[ \frac{\alpha}{\kappa \underline{T}^{1/4}} + \frac{1-\alpha}{\kappa \underline{\tau}^{1/4}} \right]^2 \vert P_{g^{T_1}} - P_{g^{T_2}}  \vert (x) \bigg\}. 
  \end{split}
 \end{equation}
  Now we estimate the difference of the densities using the representation of our steady state: 
  \begin{equation}
 \begin{split} 
  \vert \rho_{g^{T_1}} - \rho_{g^{T_2}} \vert (x) & = \Bigg\vert \int_0^1 \int_{0}^{\infty} \frac{e^{y/\kappa |v|}}{\kappa |v| (e^{1/\kappa|v|}-1)} \Big[ \rho_{g^{T_1}} (x+y) (\alpha \mathcal{M}_{T_1(x+y)}(v) + (1-\alpha)\mathcal{M}_{\tau(x+y)}(v)) - \\ & -
  \rho_{g^{T_2}}(x+y) (\alpha \mathcal{M}_{T_2(x+y)}(v) + (1-\alpha)\mathcal{M}_{\tau(x+y)}(v)) 
    \Big]   \mathrm{d}v  \mathrm{d}y
  \Bigg\vert  + \\ & 
  \Bigg\vert \int_0^1 \int_{-\infty}^{0} \frac{e^{y/\kappa |v|}}{\kappa |v| (e^{1/\kappa|v|}-1)} \Big[ \rho_{g^{T_1}} (x-y) (\alpha \mathcal{M}_{T_1(x-y)}(v) + (1-\alpha)\mathcal{M}_{\tau(x-y)}(v)) - \\ & -
  \rho_{g^{T_2}}(x-y) (\alpha \mathcal{M}_{T_2(x-y)}(v) + (1-\alpha)\mathcal{M}_{\tau(x-y)}(v)) 
    \Big]   \mathrm{d}v  \mathrm{d}y
  \Bigg\vert .
  \end{split}
 \end{equation}
 We look at the positive velocities and write 
 \begin{equation}
 \begin{split} 
& \int_0^1 \int_{0}^{\infty} \frac{e^{y/\kappa |v|}}{\kappa |v| (e^{1/\kappa|v|}-1)}
 \Big[ \rho_{g^{T_1}} (x+y)
  (\alpha \mathcal{M}_{T_1(x+y)}(v) + (1-\alpha)\mathcal{M}_{\tau(x+y)}(v)) - \\
 &\qquad \qquad \qquad \qquad- \rho_{g^{T_2}}(x+y) (\alpha \mathcal{M}_{T_2(x+y)}(v) + (1-\alpha)\mathcal{M}_{\tau(x+y)}(v)) 
    \Big]   \mathrm{d}v  \mathrm{d}y = \\ & 
     \int_0^1  (\rho_{g^{T_1}} - \rho_{g^{T_2}})(x+y) \big[(1-\alpha) \mathcal{K}_1(\tau,y) + \alpha \mathcal{K}_1(T_2,y) \big]  \mathrm{d}y
      +\\ &\qquad \qquad \qquad \qquad+ 
     \alpha  \int_0^1 \rho_{g^{T_1}}(x+y) \big[ \mathcal{K}_1(T_1,y) -  \mathcal{K}_1(T_2,y)  \big]  \mathrm{d}y 
 \end{split}
 \end{equation}
 where $\mathcal{K}_1$ is the integral kernel defined in \eqref{eq:def of F}. 
 Together with the negative velocities then 
  \begin{equation}\label{eq:est1 for contraction}
 \begin{split} 
& \vert \rho_{g^{T_1}} - \rho_{g^{T_2}} \vert (x) \lesssim  2 \alpha \|\rho_{g^{T_1}}\|_{\infty} \int_0^1 \vert \mathcal{K}_1(T_1(x+y),y) -  \mathcal{K}_1(T_2(x+y),y)  \vert + \\ 
&+ \vert \mathcal{K}_1(T_1(x-y),y) -  \mathcal{K}_1(T_2(x-y),y)  \vert   \mathrm{d}y  \\ & 
  + \int_0^1  \vert (1-\alpha) \mathcal{K}_1(\tau(x+y),y) + \alpha  \mathcal{K}_1(T_2(x+y),y)  \vert  \vert \rho_{g^{T_1}} - \rho_{g^{T_2}}  \vert (x+y)    \mathrm{d} y
   \\ & +  \int_0^1  \vert (1-\alpha) \mathcal{K}_1(\tau(x-y),y) + \alpha  \mathcal{K}_1(T_2(x-y),y)  \vert  \vert\rho_{g^{T_1}} - \rho_{g^{T_2}}  \vert (x-y)    \mathrm{d} y \\ &
  \lesssim 8 \alpha  \left[ \frac{\alpha}{\kappa \underline{T}^{1/4}} + \frac{1-\alpha}{\kappa \underline{\tau}^{1/4}} \right]^2 \times \\ 
  &\times \int_0^1 \left\vert \mathcal{K}_1(T_1(x+y),y) -  \mathcal{K}_1(T_2(x+y),y)  \right\vert
   + \left\vert \mathcal{K}_1(T_1(x-y),y) -  \mathcal{K}_1(T_2(x-y),y)  \right\vert   \mathrm{d}y  
   \\ & + 2 \|\rho_{g^{T_1}} - \rho_{g^{T_2}}  \|_{\infty} \big[ (1-\alpha) \underline{\tau}^{-1/4} + \alpha  \underline{T}^{-1/4} \big]
 \end{split}
 \end{equation}
 by the upper bound on $ \mathcal{K}_1$ in \eqref{eq:estimat on F}. 
 We notice now that $  \mathcal{K}_1$ satisfies 
 $$ \vert \mathcal{K}_1(T_1(y),y) -  \mathcal{K}_1(T_2(y),y)  \vert\leq C \vert T_1(y) - T_2(y) \vert $$ for a constant $C$ that is integrable in $y \in \mathbb{T}$. Indeed 
if  we compute 
 \begin{align*} 
 \frac{d}{dt}\mathcal{K}_1(t,y) &=  \frac{d}{dt} \int_0^{\infty} \frac{e^{-(1-y)/\kappa|v|}}{\kappa|v|(1-e^{-1/\kappa|v|})} \mathcal{M}_{T(x+y)}(v)  \mathrm{d}v \\ & 
 =  \int_0^{\infty} \frac{e^{-(1-y)/\kappa|v|}}{\kappa|v|(1-e^{-1/\kappa|v|})} \left[-\frac{1}{2t^{3/2}} \frac{e^{-v^2/2t}}{\sqrt{2\pi t}} + \frac{v^2}{2t^2}  \frac{e^{-v^2/2t}}{\sqrt{2\pi t}}   \right]  \mathrm{d}v   \\
 & = \frac{1}{2t^2} \int_0^{\infty} \frac{e^{-(1-y)/\kappa|v|}}{\kappa|v|(1-e^{-1/\kappa|v|})} \mathcal{M}_t(v) (-t+v^2)  \mathrm{d}v 
 \leq \frac{1}{2t^2} \int_0^{\infty} \frac{e^{-(1-y)/\kappa|v|}}{\kappa|v|(1-e^{-1/\kappa|v|})} \mathcal{M}_t(v) v^2  \mathrm{d}v.
 \end{align*} 
 With similar estimates as above, i.e. splitting the integral as $\int_0^{v_*} \cdots  \mathrm{d}v + \int_{v_*}^{\infty} \cdots  \mathrm{d}v$, we write 
  \begin{align*} 
   \frac{d}{dt}\mathcal{K}_1(t,y) &\lesssim \frac{1}{2t^2} \left\{ v_* + \left( 1+\frac{1}{e(1-y)}\right)\frac{1}{\kappa v_*} \right\},
  \end{align*} 
 which optimising over $v_*$, gives
 $$\frac{d}{dt} \mathcal{K}_1 (t,y) \lesssim \frac{1}{2t^2} \sqrt{\left( 1+\frac{1}{e(1-y)}\right) \frac{1}{\kappa}}. $$
 This implies that indeed 
 \begin{align} 
 \vert \mathcal{K}_1(T_1(y),y) -  \mathcal{K}_1(T_2(y),y)  \vert \leq   \frac{1}{2t^2} \sqrt{\left( 1+\frac{1}{e(1-y)}\right) \frac{1}{\kappa}}\ \vert T_1(y) - T_2(y) \vert.
 \end{align} 
 
 Inserting this then in \eqref{eq:est1 for contraction} implies 
  \begin{equation}\label{eq:est2 for contraction}
 \begin{split} 
 \vert \rho_{g^{T_1}} - \rho_{g^{T_2}} \vert (x) 
&  \lesssim 8 \alpha  \left[ \frac{\alpha}{\kappa \underline{T}^{1/4}} + \frac{1-\alpha}{\kappa \underline{\tau}^{1/4}} \right]^2
 2 \|T_1-T_2 \|_{L^{\infty}(\mathbb{T})} \int_0^1   \frac{1}{2 \underline{T}^2} \sqrt{\left( 1+\frac{1}{e(1-y)}\right) \frac{1}{\kappa}}   \mathrm{d}y  \\
  &+ 2 \|\rho_{g^{T_1}} - \rho_{g^{T_2}}  \|_{L^{\infty}(\mathbb{T})} \big[ (1-\alpha) \underline{\tau}^{-1/4} + \alpha  \underline{T}^{-1/4} \big].
 \end{split}
 \end{equation}
 There is a finite constant $C$ then so that 
  \begin{equation}\label{eq:est3 for contraction}
 \begin{split} 
 \vert \rho_{g^{T_1}} - \rho_{g^{T_2}} \vert (x) 
&  \lesssim \frac{8 C \alpha}{\sqrt{\kappa}  \underline{T}^2}  \left[ \frac{\alpha}{\kappa \underline{T}^{1/4}} + \frac{1-\alpha}{\kappa \underline{\tau}^{1/4}} \right]^2  \|T_1-T_2 \|_{L^{\infty}(\mathbb{T})}\\ & 
+ 2 \|\rho_{g^{T_1}} - \rho_{g^{T_2}}  \|_{L^{\infty}(\mathbb{T})} \big[ (1-\alpha) \underline{\tau}^{-1/4} + \alpha  \underline{T}^{-1/4} \big], 
 \end{split}
 \end{equation}
 implying that, after taking the supremum over $x$ in the left-hand side, 
\begin{equation}\label{eq:est4 for contraction}
 \begin{split} 
 \|\rho_{g^{T_1}} - \rho_{g^{T_2}}  \|_{L^{\infty}(\mathbb{T})} \Big( 1- 2 [(1-\alpha) \underline{\tau}^{-1/4} + \alpha\underline{T}^{-1/4} ] \Big) \lesssim  \frac{8 C \alpha}{\sqrt{\kappa}  \underline{T}^2}  \left[ \frac{\alpha}{\kappa \underline{T}^{1/4}} + \frac{1-\alpha}{\kappa \underline{\tau}^{1/4}} \right]^2  \|T_1-T_2 \|_{L^{\infty}(\mathbb{T})}
\end{split}
 \end{equation}
 or, as long as $\Big( 1- 2 [(1-\alpha) \underline{\tau}^{-1/4} + \alpha\underline{T}^{-1/4} ] \Big) > 0$ which is satisfied when 
 $ \underline{\tau} > \underline{T}$, $\underline{\tau}\geq 2^4$ and 
 $\alpha < \frac{2^{-1} - \underline{\tau}^{-1/4} }{ \underline{T}^{-1/4}- \underline{\tau}^{-1/4} },$ we have that 
  \begin{equation}\label{eq:est5 for contraction}
 \begin{split} 
 \|\rho_{g^{T_1}} - \rho_{g^{T_2}}  \|_{L^{\infty}(\mathbb{T})}  \lesssim  \frac{8 C \alpha}{\sqrt{\kappa}  \underline{T}^2}  
 \frac{\left[ \frac{\alpha}{\kappa \underline{T}^{1/4}} + \frac{1-\alpha}{\kappa \underline{\tau}^{1/4}} \right]^2}{\Big( 1- 2 [(1-\alpha) \underline{\tau}^{-1/4} + \alpha\underline{T}^{-1/4} ] \Big)}   \|T_1-T_2 \|_{L^{\infty}(\mathbb{T})}.
\end{split}
 \end{equation}
  
As a next step we estimate the difference of the pressures:
 \begin{equation}
 \begin{split} 
  \vert P_{g^{T_1}} - P_{g^{T_2}} \vert (x) & = 
  \Bigg\vert \int_0^1 \int_{0}^{\infty} \frac{|v| e^{y/\kappa |v|}}{\kappa (e^{1/\kappa|v|}-1)} \Big[ \rho_{g^{T_1}} (x+y) (\alpha \mathcal{M}_{T_1(x+y)}(v) + (1-\alpha)\mathcal{M}_{\tau(x+y)}(v)) - \\ & -
  \rho_{g^{T_2}}(x+y) (\alpha \mathcal{M}_{T_2(x+y)}(v) + (1-\alpha)\mathcal{M}_{\tau(x+y)}(v)) 
    \Big]   \mathrm{d}v  \mathrm{d}y
  \Bigg\vert  + \\ & 
  \Bigg\vert \int_0^1 \int_{-\infty}^{0} \frac{|v| e^{y/\kappa |v|}}{\kappa (e^{1/\kappa|v|}-1)} \Big[ \rho_{g^{T_1}} (x-y) (\alpha \mathcal{M}_{T_1(x-y)}(v) + (1-\alpha)\mathcal{M}_{\tau(x-y)}(v)) - \\ & -
  \rho_{g^{T_2}}(x-y) (\alpha \mathcal{M}_{T_2(x-y)}(v) + (1-\alpha)\mathcal{M}_{\tau(x-y)}(v)) 
    \Big]   \mathrm{d}v  \mathrm{d}y
  \Bigg\vert.
  \end{split}
 \end{equation}
Looking at the positive velocities the terms we need to estimate are
\begin{equation}
 \begin{split} 
 \int_0^1 &\Bigg\{ \alpha\int_{0}^{\infty} \frac{|v| e^{y/\kappa |v|}}{\kappa (e^{1/\kappa|v|}-1)}  \mathcal{M}_{T_1(x+y)}(v)( \rho_{g^{T_1}}- \rho_{g^{T_2}} )(x+y)  \mathrm{d}v  +\\ 
 & \qquad \qquad  +
  \alpha  \int_{0}^{\infty} \frac{|v| e^{y/\kappa |v|}}{\kappa (e^{1/\kappa|v|}-1)} \rho_{g^{T_2}}(x+y)(  \mathcal{M}_{T_1(x+y)} -  \mathcal{M}_{T_2(x+y)})(v)  \mathrm{d}v + \\ 
 &\qquad \qquad\qquad \quad+ (1-\alpha)  \int_{0}^{\infty} \frac{|v| e^{y/\kappa |v|}}{\kappa (e^{1/\kappa|v|}-1)} \mathcal{M}_{\tau(x+y)}  ( \rho_{g^{T_1}}- \rho_{g^{T_2}} )(x+y)   \mathrm{d}v \Bigg\}  \mathrm{d}y \\ 
 &  =  \int_0^1 \Bigg\{ \big( (1-\alpha) \mathcal{K}_2(\tau(x+y),y) + \alpha \mathcal{K}_2(T_1(x+y),y) \big)( \rho_{g^{T_1}}- \rho_{g^{T_2}} )(x+y) \\ & \qquad\qquad + 
  \alpha   \rho_{g^{T_2}} (x+y) \big(  \mathcal{K}_2(T_1(x+y),y) - \mathcal{K}_2(T_2(x+y),y) \big)   \Bigg\}  \mathrm{d}y
  \end{split}
 \end{equation}
where 
\begin{equation}
\mathcal{K}_2(T(x+y,y):=\int_{0}^{\infty} \frac{|v| e^{y/\kappa |v|}}{\kappa (e^{1/\kappa|v|}-1)}   \mathcal{M}_{T(x+y)} (v) \mathrm{d}v. 
 \end{equation}
 An upper bound on this integral kernel was computed above in \eqref{eq:upper bound on G}. We have then, collecting also the negative velocities, the following upper bound on the difference of the pressures: 
 \begin{equation} \label{eq:differen_of_Press1}
 \begin{split} 
  \vert P_{g^{T_1}} - P_{g^{T_2}} \vert (x) &\lesssim  \|\rho_{g^{T_1}}- \rho_{g^{T_2}} \|_{L^\infty(\mathbb{T})} 
  \int_0^1 \Big[  \vert (1-\alpha) \mathcal{K}_2(\tau(x+y),y) + \alpha \mathcal{K}_2(T_1(x+y),y) \vert \\ 
  &\qquad\qquad+  \vert (1-\alpha) \mathcal{K}_2(\tau(x-y),y) + \alpha \mathcal{K}_2(T_1(x-y),y)  \vert \Big]  \mathrm{d}y   \\ &
  + \alpha \| \rho_{g^{T_2}}\|_{L^\infty(\mathbb{T})} \int_0^1  \Big[ \vert  \mathcal{K}_2(T_1(x+y),y) - \mathcal{K}_2(T_2(x+y),y)\vert  \\ 
  & \qquad\qquad+  \vert  \mathcal{K}_2(T_1(x-y),y) - \mathcal{K}_2(T_2(x-y),y)\vert \Big]  \mathrm{d}y \\ 
  & \lesssim \frac{16 C \alpha}{\sqrt{\kappa}  \underline{T}^2}  
 \frac{\left[ \frac{\alpha}{\kappa \underline{T}^{1/4}} + \frac{1-\alpha}{\kappa \underline{\tau}^{1/4}} \right]^2}{\Big( 1- 2 [(1-\alpha) \underline{\tau}^{-1/4} + \alpha\underline{T}^{-1/4} ] \Big)}  
\left( (1-\alpha)  \frac{\sqrt{\overline{\tau}}}{\kappa} + \alpha  \frac{\sqrt{\overline{T}}}{\kappa}\right)
  \|T_1-T_2 \|_{L^{\infty}(\mathbb{T})} \\ & 
  +  4\alpha  \left[ \frac{\alpha}{\kappa \underline{T}^{1/4}} + \frac{1-\alpha}{\kappa \underline{\tau}^{1/4}} \right]^2 \int_0^1  \Big[ \vert  \mathcal{K}_2(T_1(x+y),y) - \mathcal{K}_2(T_2(x+y),y)\vert  \\ 
  & \qquad\qquad+  \vert  \mathcal{K}_2(T_1(x-y),y) - \mathcal{K}_2(T_2(x-y),y)\vert \Big]  \mathrm{d}y 
  \end{split}
 \end{equation}
 where we applied the upper bounds on $\mathcal{K}_2$, on $ \|\rho_{g^{T_1}}- \rho_{g^{T_2}} \|_{L^\infty(\mathbb{T})}$ and the constant $C$ includes the integrability of $\sqrt{1+e^{-1}/(1-y)}$. 
 
We next notice that -as in the case of $\mathcal{K}_1$-  the kernel $\mathcal{K}_2$ also satisfies $$ \vert \mathcal{K}_2(T_1(y),y) -  \mathcal{K}_2(T_2(y),y)  \vert\leq C \vert T_1(y) - T_2(y) \vert $$ for a constant $C$ that is integrable in $y \in \mathbb{T}$: for that we compute 
 \begin{align*}  
 \frac{d}{dt} \mathcal{K}_2(t,y) &= \frac{1}{t^2} \int_0^{\infty}  \frac{|v| e^{y/\kappa |v|}}{2\kappa (e^{1/\kappa|v|}-1)} \mathcal{M}_t(v) (v^2-t)   \mathrm{d}v = 
 \frac{1}{\sqrt{t}} \int_0^{\infty}   \frac{v}{2\kappa} \frac{e^{y / \kappa |v|\sqrt{t} }}{e^{1/\kappa |v|\sqrt{t}}-1} \mathcal{M}_1(v) (v^2-1) \mathrm{d}v \\ 
 & \lesssim \frac{1}{\sqrt{t}} \int_0^{\infty}   \frac{v^3}{2\kappa} \frac{e^{y / \kappa |v|\sqrt{t} }}{e^{1/\kappa |v|\sqrt{t}}-1} \mathcal{M}_1(v)  \mathrm{d}v  \lesssim 
 \frac{v_*^3}{\kappa \sqrt{t}} + \frac{1}{\kappa t} \left( 1+ \frac{e^{-1}}{1-y}\right) \int_{v_*}^{\infty}  v^3 \mathcal{M}_1(v)  \mathrm{d}v.
 \end{align*}
  This, for sufficiently large $v_*$, is less than 
   \begin{align*} 
   \frac{d}{dt} \mathcal{K}_2(t,y) &\lesssim  \frac{v_*^3}{\kappa \sqrt{t}} +  \frac{1}{\kappa t} \left( 1+ \frac{e^{-1}}{1-y}\right)\frac{(2+v_*^2)}{\sqrt{2\pi}}e^{-v_*^2/2} \lesssim \frac{v_*^3}{\kappa \sqrt{t}} +  \frac{1}{\kappa t}\left( 1+ \frac{e^{-1}}{1-y}\right) \left( \frac{2}{v_*^3}+ \frac{1}{v_*}\right).
    \end{align*}
  Taking then $v_*= \left(\frac{ 1+ \frac{e^{-1}}{1-y} }{\sqrt{t}} \right)^{1/4}$, we have 
    \begin{align*} 
   \frac{d}{dt} \mathcal{K}_2(t,y) &\lesssim \frac{1}{t^{1/8}} \frac{\left( 1+ \frac{e^{-1}}{1-y}\right)^{3/4}}{\kappa}, 
   \end{align*}
  which implies 
   \begin{align*} 
   \vert \mathcal{K}_2(T_1(y),y) -  \mathcal{K}_2(T_2(y),y)  \vert  &\lesssim \frac{1}{\kappa \underline{T}^{1/8}} \left( 1+ \frac{e^{-1}}{1-y}\right)^{3/4}  \vert T_1(y) - T_2(y) \vert. 
   \end{align*}
   Inserting this into \eqref{eq:differen_of_Press1}: 
   \begin{equation} \label{eq:differen_of_Press1}
 \begin{split} 
  \vert P_{g^{T_1}} - P_{g^{T_2}} \vert (x) &\lesssim
  \frac{16 C \alpha}{\sqrt{\kappa}  \underline{T}^2}  
 \frac{\left[ \frac{\alpha}{\kappa \underline{T}^{1/4}} + \frac{1-\alpha}{\kappa \underline{\tau}^{1/4}} \right]^2}{\Big( 1- 2 [(1-\alpha) \underline{\tau}^{-1/4} + \alpha\underline{T}^{-1/4} ] \Big)}  
\left( (1-\alpha)  \frac{\sqrt{\overline{\tau}}}{\kappa} + \alpha  \frac{\sqrt{\overline{T}}}{\kappa}\right)
  \|T_1-T_2 \|_{L^{\infty}(\mathbb{T})} \\ & 
  +  \frac{8\alpha}{\kappa}  \left[ \frac{\alpha}{\kappa \underline{T}^{1/4}} + \frac{1-\alpha}{\kappa \underline{\tau}^{1/4}} \right]^2  \frac{1}{\underline{T}^{1/8}}  \|T_1-T_2 \|_{L^{\infty}(\mathbb{T})}.
    \end{split}
 \end{equation}
   
Finally \eqref{eq1:contract_F} together with \eqref{eq:est5 for contraction} and \eqref{eq:differen_of_Press1} yield a contraction for $\mathcal{F}$: 
\begin{equation} 
\begin{split}
& \sup_{x \in \mathbb{T}} \vert \mathcal{F} (T_1)(x) -\mathcal{F} (T_2)(x) \vert \leq  C_{\alpha, \underline{T}, \underline{\tau}, \kappa} \| T_1-T_2\|_{L^{\infty}(\mathbb{T})}
\end{split} 
\end{equation}
where 
\begin{align*} 
 C_{\alpha, \underline{T}, \underline{\tau}, \kappa} 
 = 8C \alpha& \left( 1 - \Big[ \alpha \underline{T}^{-1/2}  + (1-\alpha) \underline{\tau}^{-1/2} \Big] \right)^2
\left[ \frac{\alpha}{ \underline{T}^{1/4}} + \frac{1-\alpha}{ \underline{\tau}^{1/4}} \right]^4\times 
\\ &  \qquad \qquad
 \times \Bigg\{ 
   \frac{ \left( (1-\alpha)  \sqrt{\overline{\tau}} + \alpha \sqrt{\overline{T}} \right)
 }{ \sqrt{\kappa}\underline{T}^2 \Big( 1- 2 [(1-\alpha) \underline{\tau}^{-1/4} + \alpha\underline{T}^{-1/4} ] \Big) }\left( 1+ \frac{4}{\kappa} \right) + \frac{1}{\kappa \underline{T}^{1/8}}
   \Bigg\} .
   \end{align*} 
   \end{proof} 
  \subsection{Proof of Theorem \ref{thm:existence} }
  Let us define the set 
  $S_{\underline{T}, \bar{T}} := \Big\{ \underline{T} \leq T(x) \leq  \bar{T}  \Big\} $. 
 In order to apply the Banach fixed point theorem we want to ensure that the mapping $\mathcal{F}$ maps the set $S_{\underline{T}, \bar{T}}$ into itself, i.e. 
 \begin{equation} \label{eq:inclus}
 \mathcal{F}(S_{\underline{T}, \bar{T}}) \subset S_{\underline{T}, \bar{T}}.
\end{equation} 
  
  Using the estimates from Lemmas \ref{lem: bounds_density} and \ref{lem: bounds_press}, we get the upper and lower bounds on the image of $\mathcal{F}$,  $\mathcal{F}(T)(x)  = \frac{\int_{\mathbb{T}} |v|^2 g(x,v)  \mathrm{d}v}{\int_{\mathbb{T}} g(x,v)  \mathrm{d}v} $, stated in the following lemma. 
  \begin{lemma}\label{lem:bounds on temperature}
  Under the condition \eqref{eq:constr_LB on P posit} on $\alpha$ and $\underline{\tau}$, we have uniformly in $x \in \mathbb{T}$ that
  \begin{equation}
  \begin{split} 
  &\frac{1}{4} \left[ \left[ \alpha \frac{\underline{T}^{1/2}}{2\pi} + (1-\alpha)\frac{\underline{\tau}^{1/2}}{2\pi}  \right]  
- \frac{1}{\kappa}  \left[ \frac{\alpha}{\kappa \underline{T}^{1/4}} + \frac{1-\alpha}{\kappa \underline{\tau}^{1/4}} \right]^2 \right]   \left[ \frac{\alpha}{\kappa \underline{T}^{1/4}} + \frac{1-\alpha}{\kappa \underline{\tau}^{1/4}} \right]^{-2}
  \lesssim \\ 
  &\mathcal{F}(T)(x)  \lesssim \left[ \frac{\alpha}{\kappa \underline{T}^{1/4}} + \frac{1-\alpha}{\kappa \underline{\tau}^{1/4}} \right]^2 \left[ \alpha \sqrt{\overline{T}}+ (1-\alpha)  \sqrt{\overline{\tau}} \right] \frac{1}{\left( 1 - \Big[ \alpha \underline{T}^{-1/2}  + (1-\alpha) \underline{\tau}^{-1/2} \Big] \right) }.
  \end{split} 
  \end{equation}
  \end{lemma}
  Thus in order to ensure that \eqref{eq:inclus} holds we need to assume that 
  \begin{equation}
  \begin{split} 
  &\left[ \frac{\alpha}{\kappa \underline{T}^{1/4}} + \frac{1-\alpha}{\kappa \underline{\tau}^{1/4}} \right]^2 \left[ \alpha \sqrt{\overline{T}}+ (1-\alpha)  \sqrt{\overline{\tau}} \right] \left( 1 - \Big[ \alpha \underline{T}^{-1/2}  + (1-\alpha) \underline{\tau}^{-1/2} \Big] \right)^{-1} \leq \bar{T} 
 \\&  \text{ and } \quad
  \frac{1}{4} \left[ \left[ \alpha \frac{\underline{T}^{1/2}}{2\pi} + (1-\alpha)\frac{\underline{\tau}^{1/2}}{2\pi}  \right]  
- \frac{1}{\kappa}  \left[ \frac{\alpha}{\kappa \underline{T}^{1/4}} + \frac{1-\alpha}{\kappa \underline{\tau}^{1/4}} \right]^2 \right]   \left[ \frac{\alpha}{\kappa \underline{T}^{1/4}} + \frac{1-\alpha}{\kappa \underline{\tau}^{1/4}} \right]^{-2} \geq \underline{T}. 
  \end{split} 
  \end{equation}
  By setting $\alpha = 0$ in the equations above we can see that this will be possible provided we first choose
\begin{align} \label{eq:assumption on barT, underlineT}
 \bar{T} > \frac{\sqrt{\bar{\tau}}}{\kappa^2 (\sqrt{\underline{\tau}}-1)}\ \text{ and }\
\underline{T} < \frac{\sqrt{\underline{\tau}}}{8\pi} - \frac{1}{ \kappa}. \end{align}
and then choose $\alpha$ sufficiently small depending on $\bar{\tau}, \underline{\tau}, \bar{T}, \underline{T}$ and $\kappa$.

\begin{proof}[Proof of Theorem \ref{thm:existence}]
Since $\mathcal{F}(S_{\underline{T}, \bar{T}}) \subset S_{\underline{T}, \bar{T}}$ and the mapping $\mathcal{F}$ is contractive as long as $\alpha \in \mathcal{C}_{ \underline{T}, \underline{\tau}, \kappa}$, we conclude that there is a unique fixed point, i.e. there is a unique $T \in C(\mathbb{T})$ so that $T\equiv T_g$. The uniqueness of the fixed point implies that the steady state found that satisfies the condition $\underline{T}\leq T_g(x)\leq \bar{T}$ for all $x \in \mathbb{T}$ and for $\underline{T}, \bar{T}$ satisfying \eqref{eq:assumption on barT, underlineT} with $\alpha$ sufficiently small so that $\alpha \in \mathcal{C}_{ \underline{T}, \underline{\tau}, \kappa}$, is in $(0,1)$. 
\end{proof}

\section{Linear stability}\label{sec:Linear stability}

In this section we prove the dynamical perturbative stability in the weighted $L^2$ space with weight $g(x,v)^{-1}$.

Before linearising our operator around the NESS $g$ and study its stability, 
let us provide with two more properties on $g$ that are going  to be needed for the basic estimates for the stability. That is to prove bounds on higher moments.

\subsection{Bounds on the third and fourth moments of the NESS}

Let us define the normalised third and fourth moments of the steady state $g^T$. 
\begin{defn}[normalised $3$rd and $4$th moments] 
\label{def: 3rd and 4th moment}
Let us define the third moment of the steady state $g^T$ corresponding to temperature $T$ as 
$$ d_3(x) := \frac{1}{\rho_{g^T} T^{3/2}} \int_{\mathbb{R}} v^3 g^T(x,v)  \mathrm{d}v $$
and the fourth moment  by 
$$ d_4(x):=  \frac{1}{\rho_{g^T} T^{2}} \int_{\mathbb{R}} v^4 g^T(x,v)  \mathrm{d}v - 3.$$
\end{defn}

\begin{prop} \label{prop. bounds on 3 and 4 moments} For $\alpha$ sufficiently small and $\underline{\tau} >1$, uniformly in $x \in \mathbb{T}$ we have for the third normalised moment of the non-equilibrium steady state $g$, 
 \begin{align*} 
& \frac{\bar{T}^{-3/2}}{\kappa} \Bigg\{ 
    2 \left[ \frac{\alpha}{\kappa \underline{T}^{1/4}} + \frac{1-\alpha}{\kappa \underline{\tau}^{1/4}} \right]^2 
    (\alpha \underline{T} +
      (1-\alpha) \underline{\tau} ) - 
      \frac{1}{2\sqrt{2\pi}} 
       \left[ \alpha \sqrt{\bar{T}} + (1-\alpha)\sqrt{\bar{\tau}}  \right]
     \Bigg\}
 \\ & \qquad \qquad\qquad\qquad\qquad\qquad\qquad\qquad \lesssim d_3(x) \lesssim 
  \\ & 2^{-1/4} C \Big[ \alpha \sqrt{\bar{T}} +(1-\alpha)  \sqrt{\bar{\tau}}\Big] \left[ \frac{\alpha}{\kappa \underline{T}^{1/4}} + \frac{1-\alpha}{\kappa \underline{\tau}^{1/4}} \right]^2 \Big[ 1 - \Big[ \alpha \underline{T}^{-1/2}  + (1-\alpha) \underline{\tau}^{-1/2} \Big] \Big]^{-1} \underline{T}^{-3/2}, 
 \end{align*} 
 while for the fourth normalised moment of $g$, 
 \begin{align*} 
  &\frac{ \bar{T}^{-2}}{4} \Bigg[ \left[ \frac{\alpha}{\kappa \underline{T}^{1/4}} + \frac{1-\alpha}{\kappa \underline{\tau}^{1/4}} \right]^{-2}  \sqrt{\frac{2}{\pi}}\frac{1}{\kappa} \left( \alpha \underline{T}^{3/2}+ (1-\alpha) \underline{\tau}^{3/2} \right) -  \frac{1}{\kappa^2} \left( \alpha \bar{T} + (1-\alpha) \bar{\tau} \right)
          \Bigg]-3\\ 
          & \qquad \qquad\qquad\qquad\qquad\qquad\qquad\qquad \lesssim d_4(x)  \lesssim 
          \\& \frac{4C}{\sqrt{2}} \underline{T}^{-2} \left[ 1 - \Big[ \alpha \underline{T}^{-1/2}  + (1-\alpha) \underline{\tau}^{-1/2} \Big]\right]^{-1} \left[ \frac{\alpha}{\kappa \underline{T}^{1/4}} + \frac{1-\alpha}{\kappa \underline{\tau}^{1/4}} \right]^2 \left[
    \alpha  \bar{T}^{3/4}  + 
    (1-\alpha) \bar{\tau}^{3/4} 
    \right]
 \end{align*} 
 for a finite constant $C$.
\end{prop}
  \begin{proof} 
   \emph{Upper bound on $d_3$}: 
  We start with the upper bound on $d_3$ using the representation formula for $g^T$. For positive velocities we write 
  \begin{align*} 
  \int_0^{\infty} v^3 g^T(x,v) \mathrm{d}v &= \int_0^{\infty}  \int_0^1 v^3  \frac{e^{-(1-y)/\kappa |v|}}{\kappa |v| (1-e^{-1/\kappa|v|}) } \rho_g(x+y) \left( \alpha
\mathcal{M}_{T(y+x)} + (1-\alpha) \mathcal{M}_{\tau(y+x)} \right)  \mathrm{d}y  \mathrm{d}v. 
  \end{align*}
  We estimate from above the term 
  $$\alpha \int_0^1 \rho_g(x+y) \left(  \int_0^{\infty}   \frac{|v|^2 e^{-(1-y)/\kappa }}{\kappa (1-e^{-1/\kappa|v|}) } \mathcal{M}_{T(y+x)}(v)  \mathrm{d}v  \right)  \mathrm{d}y $$
 as usual by splitting into small and large velocities and optimising. This leads for $v_* > \delta$ for small $\delta$, to the upper bound  
 \begin{align*} 
 \alpha \int_0^1\rho_g(x+y) \left[ v_* \sqrt{\frac{T}{2}} + \left(1+\frac{1}{e(1-y)} \right) \frac{\sqrt{T}}{v_*} \right]   \mathrm{d}y
  \end{align*}
  which for $v_*=\sqrt{\sqrt{2}(1+1/(e(1-y)))}$, becomes 
  $$ \alpha 2^{-1/4} \sqrt{\bar{T}} \int_0^1\rho_g(x+y) \sqrt{\left(1+\frac{1}{e(1-y)} \right) }  \mathrm{d}y \lesssim 42^{-1/4}C \sqrt{\bar{T}} \alpha   \left[ \frac{\alpha}{\kappa \underline{T}^{1/4}} + \frac{1-\alpha}{\kappa \underline{\tau}^{1/4}} \right]^2.  $$
  Eventually we have indeed, after gathering also the negative velocities and the terms with the $\tau$, 
 \begin{align*} 
 d_3(x) \lesssim 2^{-1/4} C \Big[ \alpha \sqrt{\bar{T}} +(1-\alpha)  \sqrt{\bar{\tau}}\Big] \left[ \frac{\alpha}{\kappa \underline{T}^{1/4}} + \frac{1-\alpha}{\kappa \underline{\tau}^{1/4}} \right]^2 \Big[ 1 - \Big[ \alpha \underline{T}^{-1/2}  + (1-\alpha) \underline{\tau}^{-1/2} \Big] \Big]^{-1} \underline{T}^{-3/2}
 \end{align*} 
 where in the last line we applied the lower bound on $\rho_g(x)$ to get the upper bound on $d_3(x)$. 
 
 \noindent
   \emph{Lower bound on $d_3$}:
 Regarding the lower bound on $d_3$, we use the inequality $e^{-(1-y)/\kappa |v|} \geq 1-\frac{1-y}{\kappa |v|}$ to write 
 
 \begin{align*} 
 \alpha \int_0^1 \rho_g(x+y) &\left(  \int_0^{\infty}   \frac{|v|^2 e^{-(1-y)/\kappa }}{\kappa (1-e^{-1/\kappa|v|}) } \mathcal{M}_{T(y+x)}(v)  \mathrm{d}v  \right)  \mathrm{d} y \gtrsim \\ &  
  \alpha \int_0^1 \rho_g(x+y) 
   \left(  \int_0^{\infty}   \frac{|v|^2}{\kappa} \left(1- \frac{1-y}{\kappa |v|} \right) \mathcal{M}_{T(y+x)}(v)  \mathrm{d}v  \right)
  dy \\ & =  \alpha \int_0^1 \rho_g(x+y) \left[   \frac{T(x+y)}{2\kappa} - \frac{1-y}{\kappa}\frac{\sqrt{T(x+y)}}{\sqrt{2 \pi}} \right]  \mathrm{d}y \\ 
  & \gtrsim \alpha \frac{\underline{T} }{2 \kappa} - \frac{\alpha}{2} \frac{\sqrt{\bar{T}} }{\sqrt{2\pi}\kappa} 4  \left[ \frac{\alpha}{\kappa \underline{T}^{1/4}} + \frac{1-\alpha}{\kappa \underline{\tau}^{1/4}} \right]^2.
 \end{align*} 
 This in total gives 
  \begin{align*} 
 d_3(x) &\gtrsim 
  \frac{1}{\|\rho_g\|_{L^{\infty}} \bar{T}^{3/2}} 
  \bigg\{ 
   \alpha \left(  \frac{\underline{T} }{2 \kappa} -  \frac{\sqrt{\bar{T}} }{2\sqrt{2\pi}\kappa}
    4  \left[ \frac{\alpha}{\kappa \underline{T}^{1/4}} + \frac{1-\alpha}{\kappa \underline{\tau}^{1/4}} \right]^2 \right)
    \\ &\qquad\qquad\qquad\qquad\qquad\qquad + (1-\alpha)  \left(  \frac{\underline{\tau} }{2 \kappa} - \frac{\sqrt{\bar{\tau}} }{2\sqrt{2\pi}\kappa} 4  \left[ \frac{\alpha}{\kappa \underline{\tau}^{1/4}} + \frac{1-\alpha}{\kappa \underline{\tau}^{1/4}} \right]^2 \right)
    \bigg\} 
    \\ &= 
     \frac{\bar{T}^{-3/2}}{\kappa} \Bigg\{ 
    2 \left[ \frac{\alpha}{\kappa \underline{\tau}^{1/4}} + \frac{1-\alpha}{\kappa \underline{\tau}^{1/4}} \right]^2 
    (\alpha \underline{T} +
      (1-\alpha) \underline{\tau} ) - 
      \frac{1}{2\sqrt{2\pi}} 
       \left[ \alpha \sqrt{\bar{T}} + (1-\alpha)\sqrt{\bar{\tau}}  \right]
     \Bigg\}.
   \end{align*} 
   The parameter $\alpha$ being small enough ensures for the non-negativity of the lower bound. 
   
   \noindent
   \emph{Upper bound on $d_4$}:
  We now continue with the upper bound on $d_4$. From the representation formula for $g^T$ for positive velocities we write 
  \begin{align*} 
  \int_0^{\infty} v^4 g^T(x,v) \mathrm{d}v &= \int_0^{\infty}  \int_0^1 v^3  \frac{e^{-(1-y)/\kappa }}{\kappa |v| (1-e^{-1/\kappa|v|}) } \rho_g(x+y) \left( \alpha
\mathcal{M}_{T(y+x)} + (1-\alpha) \mathcal{M}_{\tau(y+x)} \right)  \mathrm{d}y  \mathrm{d}v. 
  \end{align*}
 Then the term 
  \begin{align*}
   \alpha \int_0^1 \rho_g(x+y) \left(  \int_0^{\infty}   \frac{|v|^3 e^{-(1-y)/\kappa }}{\kappa (1-e^{-1/\kappa|v|}) } \mathcal{M}_{T(y+x)}(v)  \mathrm{d}v  \right)  \mathrm{d}y
   \end{align*}
   for the same reasons as before in $d_3$ is bounded by above as
    \begin{equation}\label{eq:upp bound d4_1}
    \begin{split} 
     &\alpha \int_0^1 \rho_g(x+y) \left[ \int_0^{v_*} |v|^3\mathcal{M}_{T(y+x)}(v)  \frac{e^{-(1-y)/\kappa }}{\kappa |v| (1-e^{-1/\kappa|v|}) }  \mathrm{d}v + 
  \int_{v_*}^{\infty} |v|^3\mathcal{M}_{T(y+x)}(v)  \frac{e^{-(1-y)/\kappa }}{\kappa |v| (1-e^{-1/\kappa|v|}) }  \mathrm{d}v 
   \right]    \mathrm{d}y \\ 
   &\lesssim   \alpha \int_0^1 \rho_g(x+y)\left[ v_* \int_0^{v_*} |v|^2\mathcal{M}_{T(y+x)}(v) \mathrm{d}v + \left( 1+\frac{e^{-1}}{1-y}\right)\int_{v_*}^{\infty} |v|^3\mathcal{M}_{T(y+x)}(v)  \mathrm{d}v
    \right]    \mathrm{d}y \\
    &\lesssim  \alpha \int_0^1 \rho_g(x+y)\left[  v_* \frac{T(x+y)}{\sqrt{2}} + \left( 1+\frac{e^{-1}}{1-y}\right)  \sqrt{\frac{T(x+y)}{2\pi}} e^{-v_*^2/(2T(x+y))}(v_*+2T(x+y)) 
     \right]    \mathrm{d}y 
     \\
    &\lesssim  \alpha \int_0^1 \rho_g(x+y)
    \left[
    v_* \frac{T(x+y)}{\sqrt{2}} + \left( 1+\frac{e^{-1}}{1-y}\right)
   \left( 
   \frac{2^3\sqrt{T(x+y)}}{\sqrt{2\pi}}
   \frac{1}{v_*} +
    \frac{2 T(x+y)^{3/2} }{\sqrt{2\pi}}
    \frac{1}{v_*^3}
   \right)
    \right]    \mathrm{d}y.
  \end{split}
  \end{equation}
  Now we optimise over $v_*$, i.e. we need to solve $v_*^4\frac{T}{\sqrt{2}} = \left( 1+\frac{e^{-1}}{1-y}\right) \left( \frac{2^3\sqrt{T}}{\sqrt{2\pi}}v_*^2 + \frac{2 T^{3/2} }{\sqrt{2\pi}} \right)$ and 
  the zeros of this polynomial are of order $T^{-1/4} \left( 1+\frac{e^{-1}}{1-y}\right)^{1/2} $. Thus we take $v_* = T^{-1/4} \left( 1+\frac{e^{-1}}{1-y}\right)^{1/2}$, 
  so that going back to \eqref{eq:upp bound d4_1} we have the upper bound
   \begin{equation}\label{eq:upp bound d4_2}
    \begin{split} 
     \alpha \int_0^1 \rho_g(x+y) &T(x+y)^{-1/4} \left( 1+\frac{e^{-1}}{1-y}\right)^{1/2} \frac{T(x+y)}{\sqrt{2}}  \mathrm{d}y 
     \\ & \leq  \frac{4 \alpha}{\sqrt{2}}  \bar{T}^{3/4} \left[ \frac{\alpha}{\kappa \underline{T}^{1/4}} + \frac{1-\alpha}{\kappa \underline{\tau}^{1/4}} \right]^2 
     \int_0^1 \left( 1+\frac{e^{-1}}{1-y}\right)^{1/2}  \mathrm{d}y
     \\ & \lesssim C  \frac{4 \alpha}{\sqrt{2}}  \bar{T}^{3/4} \left[ \frac{\alpha}{\kappa \underline{T}^{1/4}} + \frac{1-\alpha}{\kappa \underline{\tau}^{1/4}} \right]^2 
    \end{split}
  \end{equation}
  for some finite constant $C$, where we also applied the upper bound on the density $\rho_g$ given by lemma \ref{lem: bounds_density}. 
  Thus in total, together with the negative velocities and the $(1-\alpha)$ term, we have the upper bound 
   \begin{align*} 
   \int_\mathbb{R} |v|^4 g(x,v)  \mathrm{d}v \lesssim
     \frac{4C}{\sqrt{2}} \left[ \frac{\alpha}{\kappa \underline{T}^{1/4}} + \frac{1-\alpha}{\kappa \underline{\tau}^{1/4}} \right]^2 \left\{
    \alpha  \bar{T}^{3/4}  + 
    (1-\alpha) \bar{\tau}^{3/4} 
    \right\}
     \end{align*} implying together with the lower bound on the density that 
      \begin{align*} 
      d_4(x)\lesssim \underline{T}^{-2} \left[ 1 - \Big[ \alpha \underline{T}^{-1/2}  + (1-\alpha) \underline{\tau}^{-1/2} \Big]\right]^{-1}\frac{4C}{\sqrt{2}} \left[ \frac{\alpha}{\kappa \underline{T}^{1/4}} + \frac{1-\alpha}{\kappa \underline{\tau}^{1/4}} \right]^2 \left\{
    \alpha  \bar{T}^{3/4}  + 
    (1-\alpha) \bar{\tau}^{3/4} 
    \right\}.
      \end{align*}
     \noindent
   \emph{Lower bound on $d_4$}:
      Finally for the lower bound we write 
       \begin{align*} 
        \alpha \int_0^1 \rho_g(x+y) &\left(  \int_0^{\infty}   \frac{|v|^3 e^{-(1-y)/\kappa }}{\kappa (1-e^{-1/\kappa|v|}) } \mathcal{M}_{T(y+x)}(v)  \mathrm{d}v  \right)  \mathrm{d}y\\ &  \gtrsim
       \alpha \int_0^1 \rho_g(x+y) 
   \left(  \int_0^{\infty}   \frac{|v|^3}{\kappa} \left(1- \frac{1-y}{\kappa |v|} \right) \mathcal{M}_{T(y+x)}(v)  \mathrm{d}v  \right)
   \mathrm{d}y    
 \\ & \gtrsim   \alpha \int_0^1 \rho_g(x+y) \left[
\sqrt{\frac{2}{\pi}}\frac{1}{\kappa} T(x+y)^{3/2} - \frac{(1-y)}{\kappa^2}\frac{T(x+y)}{2} \right]  \mathrm{d}y 
\\ &\gtrsim  
\alpha\sqrt{\frac{2}{\pi}}\frac{1}{\kappa} \underline{T}^{3/2} -   \alpha  \left[ \frac{\alpha}{\kappa \underline{T}^{1/4}} + \frac{1-\alpha}{\kappa \underline{\tau}^{1/4}} \right]^2\frac{\bar{T}}{\kappa^2}.
        \end{align*}
  Finally, together again with the negative velocities, the $(1-\alpha)$ terms and the upper bound on the density, we write 
         \begin{align*} 
         d_4(x) \gtrsim
          &\frac{ \bar{T}^{-2}}{4} \left[ \frac{\alpha}{\kappa \underline{T}^{1/4}} + \frac{1-\alpha}{\kappa \underline{\tau}^{1/4}} \right]^{-2} 
         \Bigg[\alpha \sqrt{\frac{2}{\pi}}\frac{ \underline{T}^{3/2} }{\kappa}-  \alpha \left[ \frac{\alpha}{\kappa \underline{T}^{1/4}} + \frac{1-\alpha}{\kappa \underline{\tau}^{1/4}} 
         \right]^2\frac{\bar{T}}{\kappa^2}  
         \\ &\qquad \qquad\qquad+ (1-\alpha)\sqrt{\frac{2}{\pi}}\frac{ \underline{T}^{3/2} }{\kappa} - (1-\alpha) \left[ \frac{\alpha}{\kappa \underline{T}^{1/4}} + \frac{1-\alpha}{\kappa \underline{\tau}^{1/4}} 
         \right]^2\frac{\bar{T}}{\kappa^2}
         \Bigg] -3
         \\ & = 
         \frac{ \bar{T}^{-2}}{4} 
         \left[ \frac{\alpha}{\kappa \underline{T}^{1/4}} + \frac{1-\alpha}{\kappa \underline{\tau}^{1/4}} \right]^{-2} \Bigg[  \sqrt{\frac{2}{\pi}}\frac{1}{\kappa} \left( \alpha \underline{T}^{3/2}+ (1-\alpha) \underline{\tau}^{3/2} \right) 
          \\ & \qquad \qquad\qquad\qquad-  \left[ \frac{\alpha}{\kappa \underline{T}^{1/4}} + \frac{1-\alpha}{\kappa \underline{\tau}^{1/4}} 
         \right]^2 \frac{1}{\kappa^2} \left( \alpha \bar{T} + (1-\alpha) \bar{\tau} \right) 
         \Bigg]-3\\ 
         & =  \frac{ \bar{T}^{-2}}{4} \Bigg[ \left[ \frac{\alpha}{\kappa \underline{T}^{1/4}} + \frac{1-\alpha}{\kappa \underline{\tau}^{1/4}} \right]^{-2}  \sqrt{\frac{2}{\pi}}\frac{1}{\kappa} \left( \alpha \underline{T}^{3/2}+ (1-\alpha) \underline{\tau}^{3/2} \right) -  \frac{1}{\kappa^2} \left( \alpha \bar{T} + (1-\alpha) \bar{\tau} \right)
          \Bigg]-3.
         \end{align*}
         \end{proof}    

\subsection{Linearizing and splitting the operator}
In this subsection we linearise the operator $\mathcal{L}$ around the non-equilibrium solution $g$: $f(t,x,v) = g(x,v) + \varepsilon h(t,x,v)$. In the following lemma we state the linearised evolution equation that the new unknown $h$ satisfies. 

\begin{lemma}
The linearized equation around the steady state $g$ is
\begin{equation}\label{eq:linearised BGK}
\begin{split}
\partial_t h  = \mathfrak{L}h  &= -v\partial_x h + \frac{1}{\kappa}
\Bigg( \alpha \mathcal{M}_{T_g(x)} 
\int_{-\infty}^{\infty} \left(1 + \frac{uv}{T_g} + \frac{1}{2} \left( \frac{u^2}{T_g} -1 \right)\left( \frac{v^2}{T_g} -1 \right) \right) h(x, u) \mathrm{d} u \\ & 
\qquad\qquad\qquad\qquad\qquad\qquad\qquad\qquad\qquad+ (1-\alpha) \mathcal{M}_{\tau(x)}\int_{-\infty}^{\infty} h(x,u) \mathrm{d}u - h \Bigg)
\end{split}
\end{equation}
\end{lemma}
\begin{proof}
We write $f=g+\varepsilon h$ then we have
\[ \rho_f = \rho_g + \varepsilon \rho_h\  \text{ and }\ \rho_f u_f  = \int_{-\infty}^{\infty}  v f(x,v) \mathrm{d}v  = \rho_g u_g+ \varepsilon  \int_{-\infty}^{\infty} h(x,v) v \mathrm{d}v. \]
We remind that $u_g=0$, cf discussion just above Lem. \ref{lem: bounds_press}, and therefore by denoting $m_h:= \int_{-\infty}^{\infty} u h(x,u)\mathrm{d}u$, 
$$ u_f= \frac{\varepsilon m_h}{\rho_g (1+ \varepsilon  \rho_h/\rho_g)}= \frac{\varepsilon m_h}{\rho_g}  + \mathcal{O}(\varepsilon^2). $$
Regarding the temperature we write up to first order in $\varepsilon$
\[ \rho_f\left( T_f + u_f^2\right) = \rho_g T_g + \varepsilon P_h, \] 
and so
\begin{equation}
\begin{split} 
T_f &= \frac{T_g + \varepsilon P_h/\rho_g}{ 1+ \varepsilon\rho_h /\rho_g} - u_f^2 =
 T_g +\varepsilon \left(\frac{P_h}{\rho_g} - \frac{\rho_h}{\rho_g}T_g \right) - \left( u_g + \varepsilon \left( \frac{m_h}{\rho_g} - \frac{\rho_h u_g}{\rho_g} \right)\right)^2 + o(\varepsilon) \\ & = 
T_g + \varepsilon\left(\frac{P_h}{\rho_g} - \frac{\rho_h}{\rho_g}T_g \right) + o(\varepsilon) = 
 T_g\left( 1+ \varepsilon \left(\frac{P_h}{P_g} - \frac{\rho_h}{\rho_g} \right)  \right) + o(\varepsilon) \\ & = 
T_g \left( 1+ \frac{\varepsilon}{\rho_g}\int_{-\infty}^{\infty}\left( \frac{v^2}{T_g}-1 \right) h(x,v) \mathrm{d}v  \right) + o(\varepsilon)
\end{split} 
\end{equation} 
where we used again that $u_g = 0$.
Next we move to the linearisation of the Maxwellian. Up to first order in $\varepsilon$ we write
\begin{equation}
\begin{split} 
 |v-u_f|^2/2T_f &= \left|v-\varepsilon \left(\frac{m_h}{\rho_g}\right)\right|^2 \frac{1}{2T_g} \left[ 1 - \frac{\varepsilon}{\rho_g} \int_{-\infty}^{\infty} \left[\frac{u^2}{T_g}-1 \right] h(x,u) \mathrm{d}u \right] + o(\varepsilon) \\ 
 \\ & = \left[ \frac{v^2}{2T_g} -  \varepsilon v \left( \frac{m_h}{T_g \rho_g}  \right)\right] \left[ 1 - \frac{\varepsilon}{\rho_g} \int_{-\infty}^{\infty} \left[\frac{u^2}{T_g}-1 \right] h(x,u) \mathrm{d}u \right] + o(\varepsilon)\\
 &= \frac{v^2}{2T_g}  - \frac{\varepsilon v^2}{2 T_g \rho_g} \int_{-\infty}^{\infty} \left[\frac{u^2}{T_g}-1 \right] h(x,u) \mathrm{d}u  - \frac{\varepsilon v m_h}{T_g \rho_g} + o(\varepsilon)\\
\end{split} 
\end{equation} 
and so by expanding the exponential near $0$, 
\begin{equation}
\begin{split} 
\exp \left( - \frac{|v-u_f|^2}{2T_f} \right) =\exp \left( -\frac{v^2}{2T_g} \right) \left(1+ \frac{\varepsilon v^2}{2 T_g \rho_g} \int_{-\infty}^{\infty} \left[\frac{u^2}{T_g}-1 \right] h(x,u) \mathrm{d}u  +   \frac{\varepsilon v m_h}{T_g \rho_g}   \right). 
\end{split} 
\end{equation} 
For its normalisation we have by expanding again around $0$ that 
\begin{equation}
\begin{split} 
(2\pi T_f)^{-1/2} &= \frac{1}{\sqrt{2\pi T_g} } \left( 1 - \frac{\varepsilon}{2} \left( \frac{P_h}{P_g} - \frac{\rho_h}{\rho_g}\right) \right) + o(\varepsilon) \\ & 
=\frac{1}{\sqrt{2\pi T_g} }  \left( 1-  \frac{\varepsilon}{2\rho_g}  \int_{-\infty}^{\infty} \left[\frac{u^2}{T_g}-1 \right] h(x,u) \mathrm{d}u  \right) 
+ o(\varepsilon).
\end{split} 
\end{equation} 
Therefore altogether the infinitesimal Maxwellian is
\begin{equation}
\begin{split} 
 \mathcal{M}_{u_f,T_f} (v) &= \mathcal{M}_{0,T_g} (v) \bigg( 1 +  \frac{\varepsilon v^2}{2 T_g \rho_g} \int_{-\infty}^{\infty} \left[\frac{u^2}{T_g}-1 \right] h(x,u) \mathrm{d}u + \frac{\varepsilon v }{T_g \rho_g}  \int_{-\infty}^{\infty} u h(x,u) \mathrm{d}u \\& \qquad \qquad \qquad \qquad\qquad \qquad\qquad \qquad
 -  \frac{\varepsilon}{2\rho_g}  \int_{-\infty}^{\infty} \left[\frac{u^2}{T_g}-1 \right] h(x,u) \mathrm{d}u\bigg)+ o(\varepsilon)
 \\ &  = 
  \mathcal{M}_{0,T_g} (v) \bigg( 1 +  \frac{\varepsilon }{2 \rho_g}  \left[\frac{v^2}{T_g}-1 \right] \int_{-\infty}^{\infty} \left[\frac{u^2}{T_g}-1 \right] h(x,u) \mathrm{d}u+ \frac{\varepsilon v }{T_g \rho_g}  \int_{-\infty}^{\infty} u h(x,u) \mathrm{d}u 
  \bigg)+ o(\varepsilon), 
\end{split} 
\end{equation} 
and so the whole nonlinear term is 
$$\rho_f  \mathcal{M}_{u_f,T_f} (v) = \rho_g \mathcal{M}_{0,T_g} (v) \bigg[ 1+
 \frac{\varepsilon}{\rho_g} \left\{ \int_{-\infty}^{\infty}  \left(1+ \frac{1}{2} \left[ \frac{v^2}{T_g}-1 \right]\left[\frac{u^2}{T_g}-1 \right] + \frac{uv}{T_g}  \right) h(x,u)
 \mathrm{d}u   \right\}
\bigg].$$
Using then that $-v\partial_x g + \mathcal{L}g=0$, we end up with the stated operator $\mathfrak{L}$.
\end{proof}

Next we stress that our linearised operator is a non self-adjoint operator in the weighted space $L^2(g^{-1})$. Nevertheless, we proceed by decomposing the operator $\mathfrak{L} = \mathcal{T} + \mathcal{C}$ into a symmetric and an antisymmetric part. This is motivated by techniques in hypocoercivity, see for example \cite{Villani09,Herau06, DMS15} and more recently a technique via Schur complements, \cite{BFLS20}.
\begin{lemma} \label{lem:splitting sym antisym}
The linearised equation \eqref{eq:linearised BGK} has the form
\[ \partial_t f + \mathcal{T}f = \mathcal{C}f \] where $\mathcal{C}$ is symmetric in $L^2(g^{-1})$ and $\mathcal{T}$ is antisymmetric. The explicit formulas are given by
\begin{equation}\label{eq:symmetric part}
\begin{split}
\mathcal{C}&f = \frac{\alpha}{2 \kappa} \int_{-\infty}^{\infty} \left(\mathcal{M}_{T_g}(v) + \frac{\mathcal{M}_{T_g}(u)}{g(x,u)} g(x,v) \right) \left( 1 + \frac{uv}{T_g} + \frac{1}{2} \left[ \frac{u^2}{T_g} -1 \right]\left[\frac{v^2}{T_g} -1 \right] \right) f(x,u) \mathrm{d}u\\
&  + \frac{1-\alpha}{2 \kappa} \int_{-\infty}^{\infty} \left( \mathcal{M}_{\tau}(v) + \frac{\mathcal{M}_{\tau}(u)}{g(x,u)} g(x,v) \right) f(x,u) \mathrm{d} u - \frac{1}{2 \kappa} \left( 1 + \rho_g \frac{\alpha \mathcal{M}_{T_g} + (1-\alpha) \mathcal{M}_\tau}{g(x,v)} \right) f(x,v), 
\end{split}
\end{equation}
while 
\begin{equation}\label{eq:antisymmetric part}
\begin{split}
\mathcal{T}f &= -v \partial_x f + \frac{\alpha}{2\kappa} \int_{-\infty}^{\infty}  \left( \mathcal{M}_{T_g}(v) - \frac{\mathcal{M}_{T_g}(u)g(x,v)}{g(x,u)} \right) \left( 1+ \frac{uv}{T_g} + \frac{1}{2} \left[ \frac{u^2}{T_g}-1\right]\left[ \frac{v^2}{T_g} -1 \right] \right)f(x,u) \mathrm{d}u \\
& \quad + \frac{1-\alpha}{2\kappa} \int_{-\infty}^{\infty}  \left( \mathcal{M}_{\tau(x)}(v) - \frac{\mathcal{M}_{\tau(x)} (u) g(x,v)}{g(x,u)} \right) f(x,u) \mathrm{d}u \\
& \quad - \frac{1}{2 \kappa} f(x,v) \left( 1 - \frac{\rho_g(x)}{g(x,v)} \left( \alpha \mathcal{M}_{T_g(x)}(v) + (1-\alpha) \mathcal{M}_{\tau(x)}(v) \right)  \right).
\end{split}
\end{equation}
\end{lemma}
\begin{proof}
We stress here that we denote by $\mathfrak{L}$ the whole operator, together with the transport part so that 
\begin{align*}
&\mathfrak{L}  h =   
-v \partial_x h + \frac{1}{\kappa}
\Bigg( \alpha \mathcal{M}_{T_g(x)} \int_{-\infty}^{\infty} \left(1 + \frac{uv}{T_g} + \frac{1}{2} \left[ \frac{u^2}{T_g} -1 \right]\left[ \frac{v^2}{T_g} -1 \right] \right) h(x, u) \mathrm{d} u
\\ &  \qquad\qquad\qquad\qquad\qquad\qquad\qquad\qquad\qquad
+ (1-\alpha) \mathcal{M}_{\tau(x)}\int_{-\infty}^{\infty} h(x,u) \mathrm{d}u - h \Bigg).
\end{align*}
Calculating the adjoint of this operator $\mathfrak{L}^*$ in $L^2(g^{-1})$ we find that it is
\begin{equation}\label{eq:adjoint_linearised}
\begin{split}
\mathfrak{L}^*h &= v\partial_x h \\  &+ \frac{1}{\kappa} \left[ \alpha  g(x,v)  \int_{-\infty}^{\infty} \frac{\mathcal{M}_{T_g}(u)}{g(x,u)}\left( 1 + \frac{uv}{T_g} + \frac{1}{2} \left[ \frac{u^2}{T_g} -1 \right] \left[ \frac{v^2}{T_g} -1 \right] \right) h(x,u) \mathrm{d}u \right.\\
& \quad \left.+ (1-\alpha) \int_{-\infty}^{\infty} \frac{\mathcal{M}_{\tau(x)}(u)}{g(x,u)} h(x,u) \mathrm{d}u g(x,v)\right]\\
& \quad - \frac{\rho_g(x)}{g(x,v)}\left( \frac{\alpha}{\kappa} \mathcal{M}_{T_g(x)}(v) + \frac{(1-\alpha)}{\kappa} \mathcal{M}_{\tau(x)}(v) \right) h(x,v).
\end{split}
\end{equation}
Indeed the adjoint of the transport part gives, using that $g$ is NESS:
\begin{align*}
(-v\partial_x)^*h(x,v) &= v\partial_xh(x,v) -v\frac{\partial_x g(x,v)}{g(x,v)}h(x,v) \\ 
&=v\partial_xh(x,v) - g(x,v)^{-1}\rho_g(x)(\alpha \mathcal{M}_{T_g(x)}(v)  + (1-\alpha) \mathcal{M}_{\tau(x)}(v))h(x,v) -h(x,v),
\end{align*}
which accounts for the first and the last term in \eqref{eq:adjoint_linearised}.
Then we use this to get symmetric and antisymmetric parts since $\mathcal{C} = \frac{\mathfrak{L}+\mathfrak{L}^*}{2}$ and $\mathcal{T} =\frac{\mathfrak{L}-\mathfrak{L}^*}{2} $.
\end{proof}
Typically in $L^2$-hypocoercivity, see \cite{DMS15}, the strategy is to prove that the collisional symmetric operator in $L^2(g^{-1})$ satisfies a  \emph{microscopic coercivity} property on the orthogonal complement of its null space and that the transport antisymmetric operator satisfies a \emph{macroscopic coercivity} property exactly on the null space of the collisional operator. Then after introducing a modified entropy functional, equivalent to $L^2$ norm, under additional general assumptions on the decomposing operators, one can control its dissipation and conclude exponential relaxation of the semigroup to the steady state. 

This general framework in $L^2$ to estimate relaxation rates for a general class of kinetic equations does not apply in our case as it is, because our steady state does not have an explicit formula. In \cite{DMS15} it is required that the global equilibrium lies in the intersection of the null spaces of the two decomposing operators, i.e. the transport and the collision operators. This is not the case in our out-of-equilibrium setting. 

To our knowledge there are two articles so far dealing with such circumstances, \cite{EmFrCl17} where the authors treat the problem perturbatively and \cite{CalvRaoulSchm15} where the transport and collision operators  are redefined. 

Our approach is mainly inspired by the latter: in particular we split the operator differently rather than in \cite{DMS15}, into a symmetric and to a skew symmetric part as is shown in Lemma \ref{lem:splitting sym antisym}. Then we show a microscopic coercivity inequality for the symmetric part and a macroscopic coercivity for the antisymmetric part. 

Before we move on with proving such properties, let us stress that now one can check with explicit computations that $$\mathcal{C} g = \mathcal{T} g=0,$$ where $\mathcal{C}, \mathcal{T}$ are computed explicitly in Lemma \ref{lem:splitting sym antisym}. 
 So with this decomposition the steady state belongs in the intersection of the kernels of the decomposed part of the operator.

We denote by $\Pi$ the orthogonal projection to $\operatorname{Ker}(\mathcal{C})$ which is given by 
$$\Pi f = \rho_f \frac{g}{\rho_g}, \text{ where } \rho_f (x) = \int_{-\infty}^{\infty} f(x,v) \mathrm{d}v. $$
The space that we work in is 
$$\mathcal{H} = \Big\{ f \in L^2 \left( \frac{\mathrm{d}x \mathrm{d}v}{g} \right): \iint_{\mathbb{T}\times \mathbb{R}} f(x,v)\mathrm{d}x \mathrm{d}v=0  \Big\} $$
induced by the scalar product
\begin{equation} \label{eq:scalar prod}
 \langle f_1, f_2 \rangle= \langle f_1, f_2 \rangle_{L_{x,v}^2(g^{-1})} = \iint f_1(x,u)f_2(x,u)/g(x,u) \mathrm{d}u \mathrm{d}x.
 \end{equation} 


\subsection{Microscopic Coercivity} \label{subsec: micro coerc}
 This subsection is devoted to the coercivity property of the symmetric operator on the orthogonal of the kernel of $\mathcal{C}$.
\begin{prop}[Microscopic coercivity] \label{prop:micro coercivity}
With the above notation, we have that the operator $\mathcal{C}$ satisfies 
$$ -\langle \mathcal{C}h ,h \rangle \geq \lambda_m \|(I-\Pi) h \|^2$$
for some positive constant $\lambda_m$. 
\end{prop}
\begin{proof} 
In this proof we will write  $$\langle f_1, f_2 \rangle  = \langle f_1, f_2 \rangle_{L^2(g^{-1})} = \int f_1(u)f_2(u)/g(x,u) \mathrm{d}u$$ which we note is a function of $x$.

We start by computing $\langle \mathcal{C} h, h \rangle_{L^2(g^{-1})}$. In order to do this let us rewrite
\begin{align*}
 \mathcal{C} h 
 &= \frac{\alpha}{2\kappa}  \mathcal{M}_{T_g}
  \left(\langle h, g \rangle + \frac{v}{\sqrt{T_g}} \left\langle h, \frac{u}{\sqrt{T_g}} g \right\rangle + \frac{1}{\sqrt{2}} \left[ \frac{v^2}{T_g} -1 \right] \left\langle h , \frac{1}{\sqrt{2}}\left[ \frac{u^2}{T_g} -1 \right] g \right\rangle \right)\\
& \quad + 
\frac{\alpha}{2\kappa} g(x,v) \left( \langle h, \mathcal{M}_{T_g} \rangle + \frac{v}{\sqrt{T_g}} \left\langle h , \frac{u}{\sqrt{T_g}}\mathcal{M}_{T_g} \right\rangle + \frac{1}{\sqrt{2}} \left[ \frac{v^2}{T_g} -1 \right] 
\left\langle h, \frac{1}{\sqrt{2}} \left[ \frac{u^2}{T_g} -1 \right] \mathcal{M}_{T_g} \right\rangle \right)\\
& \quad + 
\frac{1-\alpha}{2 \kappa} \left( g(x,v)  \langle h, \mathcal{M}_\tau \rangle + \mathcal{M}_\tau \langle h, g \rangle \right) \\
& \quad - \frac{1}{2 \kappa}\left( 1 + \rho_g \frac{\alpha \mathcal{M}_{T_g} + (1-\alpha) \mathcal{M}_\tau}{g(x,v)} \right) h(x,v).
\end{align*}
Using this we write
\begin{align*}
\langle \mathcal{C} h, h \rangle &= \frac{\alpha}{\kappa} \Bigg( \left\langle h, \mathcal{M}_{T_g} \right\rangle \left\langle h, g \right\rangle + \left\langle h, \frac{u}{\sqrt{T_g}} \mathcal{M}_{T_g} \right\rangle \left\langle h, \frac{u}{\sqrt{T_g}} g \right\rangle \\ &
\qquad\qquad\qquad+ \left\langle h, \frac{1}{\sqrt{2}} \left[\frac{u^2}{T_g} -1\right] \mathcal{M}_{T_g} \right\rangle \left\langle h, \frac{1}{\sqrt{2}} \left[ \frac{u^2}{T_g} -1 \right] g \right\rangle 
\Bigg) + \frac{1-\alpha}{\kappa} \langle h, \mathcal{M}_\tau\rangle \langle h, g \rangle \\
& \qquad\qquad\qquad\qquad\qquad  - \frac{1}{2 \kappa} \int_{-\infty}^{\infty} \left( 1 + \rho_g \frac{\alpha \mathcal{M}_{T_g} + (1-\alpha) \mathcal{M}_\tau}{g(x,v)} \right) \frac{h(x,v)^2}{g(x,v)} \mathrm{d}v.
\end{align*}
We split this into
\[
 \langle  \mathcal{C}  h , h \rangle = \frac{1}{\kappa}\left[ \alpha  \mathcal{I} + (1-\alpha) \mathcal{J} \right]
  \] 
 where
\begin{align*} 
\mathcal{I} &= \left\langle h, \mathcal{M}_{T_g} \right\rangle \left\langle h, g \right\rangle + \left\langle h, \frac{u}{\sqrt{T_g}} \mathcal{M}_{T_g} \right\rangle \left\langle h, \frac{u}{\sqrt{T_g}} g \right\rangle + \left\langle h, \frac{1}{\sqrt{2}} \left[\frac{u^2}{T_g} -1\right] \mathcal{M}_{T_g} \right\rangle \left\langle h, \frac{1}{\sqrt{2}} \left[ \frac{u^2}{T_g} -1 \right] g \right\rangle \\ 
& -
 \frac{1}{2} \int_{-\infty}^{\infty} \left( 1 + \rho_g(x) \frac{\mathcal{M}_{T_g}}{g} \right) \frac{h(x,u)^2}{g(x,u)}  \mathrm{d}u
\end{align*} while 
\[  \mathcal{J} = \langle h, \mathcal{M}_{\tau} \rangle \langle h, g \rangle - 
\frac{1}{2} \int_{-\infty}^{\infty} \left(1+ \rho_g(x) \frac{\mathcal{M}_{\tau(x)} }{g} \right) \frac{h(x,u)^2}{g(x,u)} \mathrm{d}u
\]
and we are going to bound them separately. 

Motivated by this, we consider as a function of $v$ (for a fixed $x \in \mathbb{T}$) the functions:
\[p_1(v)g:=  \frac{g}{\sqrt{\rho_g}},\ p_2(v)g :=\frac{vg}{\sqrt{\rho_g T_g}},\ p_3(v)g := \frac{\left( \frac{v^2}{T_g} - \frac{d_3(x) v}{\sqrt{T_g}} -1 \right) g}{\sqrt{\rho_g(2+d_4 -d_3^2)}}.
\] 
We observe that these are orthogonal functions in $L_v^2(g^{-1})$.  

Then we decompose $h$ as the sum of $h = \hat{h} + \tilde{h}$ where $\tilde{h}$ is orthogonal in $L_v^2(g^{-1})$ to the space spanned by $p_1(v)g, p_2(v)g, p_3(v)g$, while $\hat{h}$ belongs in the space that these quantities span. 
We collect the above observations in the following lemma. 
\begin{lemma}\label{lem: orthonormal}
For fixed $x \in \mathbb{T}$, we consider the mappings $v \mapsto p_i(v) g(x,v) $ for $i=1,2,3$  where
\[p_1(v)g:=  \frac{g}{\sqrt{\rho_g}},\ p_2(v)g :=\frac{vg}{\sqrt{\rho_g T_g}},\ p_3(v)g := \frac{\left( \frac{v^2}{T_g} - \frac{d_3(x) v}{\sqrt{T_g}} -1 \right) g}{\sqrt{\rho_g(2+d_4 -d_3^2)}}.\]  These are orthogonal functions in $L^2_v(g^{-1})$ of norm $1$.
\end{lemma}
\begin{proof}
This is just a computation.
\begin{align*}
\int \left( \frac{v^2}{T_g} - \frac{d_3 v}{\sqrt{T_g}} - 1 \right)^2 g\mathrm{d}v &= \frac{\rho_g T_g^2 (3+d_4)}{T_g} - \frac{2\rho_g d_3^2T_g^{3/2}}{T_g^{3/2}} + \frac{\rho_g d_3^2 T_g}{T_g} -2 \frac{\rho_g T_g}{T_g} + \rho_g  \\ & = \rho_g \left( 2+d_4 -d_3^2\right).
\end{align*}
\end{proof}

We continue by writing $h =\big[  a p_1 g + \beta p_2 g + \gamma p_3 g \big] + \tilde{h}$ for some constants $a,\beta,\gamma$ and where  $\tilde{h}$ is orthogonal to the space spanned by $p_1g, p_2g, p_3 g$ in $L^2_v(g^{-1})$. In the following lemma we collect the explicit values of each term appearing  in $\mathcal{I}$, $ \mathcal{J}$. One can check these by explicit computations. 
\begin{lemma}
Decomposing the solution of \eqref{eq:linearised BGK} as $h = a p_1 g + \beta p_2 g + \gamma p_3 g + \tilde{h}$, for some constants $a,\beta,\gamma$. Then 
\[ \langle h, g \rangle = a\sqrt{\rho_g},
 \quad \left\langle h, \frac{u}{\sqrt{T_g}} g \right\rangle  =\beta \sqrt{\rho_g} ,
 \quad \left\langle h, \frac{1}{\sqrt{2}}\left[ \frac{u^2}{T_g} -1 \right] g \right\rangle  = \gamma \sqrt{ \frac{\rho_g (2+d_4 -d_3^2)}{2}} + \beta d_3 \sqrt{\rho_g} \]
and
\[ \left\langle h, \mathcal{M}_{T_g} \right\rangle  =\frac{a}{\sqrt{\rho_g}} + \left\langle \tilde{h}, \mathcal{M}_{T_g} \right\rangle , 
\quad \left\langle h, \frac{u}{\sqrt{T_g}}\right\rangle  = \frac{\beta}{\sqrt{\rho_g}} - \frac{\gamma d_3}{\sqrt{\rho_g(2+d_4 - d_3^2)}} + \left\langle \tilde{h}, \frac{u}{\sqrt{T_g}} \mathcal{M}_{T_g} \right\rangle    \]
\end{lemma}

From now on we will denote by $Z=Z(x) := 2+d_4(x)-d_3(x)^2$ and let us note that we are able to bound this from below and above uniformly in $x$, thanks to the Proposition \ref{prop. bounds on 3 and 4 moments} with constants depending on $\alpha, \underline{\tau}, \bar{\tau}$. 
Inserting these values into $\mathcal{I}$, for its first three terms we have 
\begin{align*}
&\left\langle  h, \mathcal{M}_{T_g} \right\rangle \left\langle  h, g \right\rangle +
 \left\langle  h, \frac{u}{\sqrt{T_g}} \mathcal{M}_{T_g} \right\rangle \left\langle  h, \frac{u}{\sqrt{T_g}} g \right\rangle 
 +\left\langle  h, \frac{1}{\sqrt{2}} \left[\frac{u^2}{T_g} -1\right] \mathcal{M}_{T_g} \right\rangle \left\langle  h, \frac{1}{\sqrt{2}} \left[ \frac{u^2}{T_g} -1 \right] g \right\rangle \\
&= a^2 + \beta^2 +\gamma^2 + \\ & \qquad\quad+
 \sqrt{\rho_g} \left[ a \left\langle \tilde{h}, \mathcal{M}_{T_g} \right\rangle  +
  \beta \left\langle \tilde{h}, \frac{v}{\sqrt{T_g}} \mathcal{M}_{T_g} \right\rangle + 
  \frac{\gamma \sqrt{Z} + \beta d_3}{\sqrt{2}} \left\langle \tilde{h}, \frac{1}{\sqrt{2}} \left[ \frac{v^2}{T_g} - 1 \right] \mathcal{M}_{T_g} \right\rangle \right].
\end{align*}

Moving to the fourth term in $\mathcal{I}$,we first note that we have
\begin{equation}\label{eq: micros coerc est1}
\int \frac{h^2}{g} = a^2 + \beta^2 +\gamma^2 + \int \frac{\tilde{h}^2}{g}
\end{equation}
and therefore 
\begin{align*}
\int \frac{h^2}{g^2} \rho_g \mathcal{M}_{T_g} &=\int \frac{\hat{h}^2}{g^2} \rho_g \mathcal{M}_{T_g} + \int \frac{\tilde{h}^2}{g^2} \rho_g \mathcal{M}_{T_g} = 
  a^2 + \beta^2 +  \frac{\gamma^2(2+d_3^2)}{Z} - \frac{2\beta \gamma}{\sqrt{Z}} d_3 \\
&  + \sqrt{\rho_g} 
\left[
 2a \langle \tilde{h}, \mathcal{M}_{T_g} \rangle
  +2 \left( \beta -\frac{ \gamma d_3}{\sqrt{Z}} \right) \left\langle \tilde{h}, \frac{v}{\sqrt{T_g}} \mathcal{M}_{T_g} \right\rangle + 
  2 \gamma \sqrt{\frac{2}{Z}} 
  \left\langle \tilde{h}, \frac{1}{\sqrt{2}} \left[ \frac{v^2}{T_g} -1 \right] \mathcal{M}_{T_g} 
  \right\rangle \right]\\
& +  \int_{-\infty}^{\infty}  \frac{\tilde{h}(x,v)^2}{g(x,v)^2}\rho_g(x) \mathcal{M}_{T_g}(v) \mathrm{d}v.
\end{align*}

Putting this together we get
\begin{equation} \label{eq:est1 on I}
\begin{split}
\mathcal{I}& \leq 
a^2+ \beta^2+\gamma^2 + a \sqrt{\rho_g} \left\langle \tilde{h}, \mathcal{M}_{T_g} \right\rangle  + 2\beta \sqrt{\rho_g}\left\langle \tilde{h}, \frac{v}{\sqrt{T_g}} \mathcal{M}_{T_g} \right\rangle \\ 
&+   \left\langle \tilde{h},  \left[ \frac{v^2}{T_g} -1 \right] \mathcal{M}_{T_g} 
  \right\rangle \sqrt{\frac{\rho_g}{2}}\left( \gamma \sqrt{\frac{Z}{2}} + \frac{\beta d_3}{\sqrt{2}}  \right) 
   - \frac{a^2}{2}-\frac{\beta^2}{2} - \frac{\gamma^2(2+d_3^2)}{2Z} \\ &- \frac{\beta\gamma}{\sqrt{Z}} d_3 - \sqrt{\rho_g}\left( a  \left\langle \tilde{h}, \mathcal{M}_{T_g} \right\rangle +
    \left( \beta - \frac{\gamma d_3 }{\sqrt{Z}}  \right)  \left\langle \tilde{h}, \frac{v}{\sqrt{T_g}} \mathcal{M}_{T_g} \right\rangle     
    + \gamma \sqrt{\frac{2}{Z}} \left\langle \tilde{h}, \frac{1}{\sqrt{2}} \left[ \frac{v^2}{T_g} -1 \right] \mathcal{M}_{T_g} 
  \right\rangle
      \right) \\& 
     - \frac{1}{2} \int_{-\infty}^{\infty}  \frac{\tilde{h}(x,v)^2}{g(x,v)^2}\rho_g(x) \mathcal{M}_{T_g}(v) \mathrm{d}v
       -\frac{a^2}{2} - \frac{\beta^2}{2} - \frac{\gamma^2}{2} - \frac{1}{2}\int_{-\infty}^{\infty} 
       \frac{\tilde{h}(x,v)^2}{g(x,v)} \mathrm{d}v
\\ &\leq
 \frac{\gamma^2}{2} \left( 1- \frac{2+d_3^2}{2+d_4 - d_3^2} \right) + 
 \frac{\beta \gamma}{\sqrt{Z}} |d_3|\\
& \quad + \sqrt{\rho_g} \left[  \frac{\gamma d_3}{\sqrt{Z}} \left\langle \tilde{h}, \frac{v}{\sqrt{T_g}} \mathcal{M}_{T_g} \right\rangle + \left(\frac{\gamma}{\sqrt{2}} \left( \sqrt{\frac{Z}{2}} - \sqrt{ \frac{2}{Z}} \right) + \frac{\beta d_3}{\sqrt{2}}\right)  \left\langle \tilde{h},\left[ \frac{v^2}{T_g}-1 \right] \mathcal{M}_{T_g} \right\rangle \right] \\
& \quad - \frac{1}{2} \int \frac{\tilde{h}^2}{g} \left( 1+ \rho_g \frac{\mathcal{M}_{T_g}}{g} \right).
\end{split}
\end{equation}
Now regarding the terms involving $\tilde{h}$ we first estimate by Cauchy-Schwartz and then use that $\int_{-\infty}^{\infty} v^2\mathcal{M}_{T_g}(v) \mathrm{d}v =T_g$: 
\begin{equation} \label{eq:est2 on I}
\begin{split}
\sqrt{\rho_g(x)} \left\langle \tilde{h}, \frac{v}{\sqrt{T_g}} \mathcal{M}_{T_g} \right\rangle &= 
\sqrt{\rho_g(x)} \int_{-\infty}^{\infty} \tilde{h}(x,v) \frac{v}{\sqrt{T_g}} \frac{\mathcal{M}_{T_g}(v)}{g(x,v)} \mathrm{d}v \\
&\leq \left( \int_{-\infty}^{\infty} \frac{\tilde{h}(x,v)^2}{g(x,v)^2} \rho_g(x) \mathcal{M}_{T_g}(v) \mathrm{d}v \right)^{1/2} 
\left( \int_{-\infty}^{\infty} \frac{v^2}{T_g} \mathcal{M}_{T_g} (v) \mathrm{d}v \right)^{1/2} 
\\ & =\left( \int_{-\infty}^{\infty} \frac{\tilde{h}^2}{g^2} \rho_g(x) \mathcal{M}_{T_g} (v) \mathrm{d}v \right)^{1/2}, 
\end{split}
\end{equation}
and similarly
\begin{equation}\label{eq:est3 on I}
\begin{split}
\sqrt{\rho_g(x)} \left\langle \tilde{h}, \frac{1}{\sqrt{2}}\left[ \frac{v^2}{T_g} -1 \right] \mathcal{M}_{T_g}  \right\rangle  \leq \left(\int_{-\infty}^{\infty} \frac{\tilde{h}^2}{g^2} \rho_g(x) \mathcal{M}_{T_g}(v) \mathrm{d}v \right)^{1/2}. 
\end{split}
\end{equation}
The first line in the right-hand side of  \eqref{eq:est1 on I} becomes $ \frac{\gamma^2}{2} \frac{d_4-2d_3^2 }{2+d_4 -d_3^2}  + \frac{\beta\gamma d_3}{\sqrt{Z}}$. We also gather the other terms together combining them with the estimates in \eqref{eq:est2 on I}, \eqref{eq:est3 on I} and apply Young's inequality to get 
\begin{align*}
\mathcal{I} & \leq \gamma^2 \frac{d_4 - 2d_3^2}{2+d_4 - d_3^2} + \frac{\beta\gamma d_3}{\sqrt{Z}} + \frac{1}{4} \left( \frac{d_3 |\gamma|}{\sqrt{Z}} + \left| \gamma \left( \sqrt{\frac{Z}{2}} - \sqrt{\frac{2}{Z}} \right) + \frac{\beta d_3}{\sqrt{2}} \right| \right)^2 \\
& \leq \varphi(d_3,d_4) (\beta^2 +\gamma^2) \leq \varphi(d_3,d_4) \int_{-\infty}^{\infty} \frac{h^\perp(x,v)^2}{g(x,v)}\mathrm{d}v, 
\end{align*}
where $h^\perp:=(I-\Pi)h = h - \left\langle h, \frac{g}{\sqrt{\rho_g}} \right\rangle \frac{g}{\sqrt{\rho_g}} = 
 h - \left\langle h, p_1 g \right\rangle  \frac{g}{\sqrt{\rho_g}}$.
  And $\varphi(d_3, d_4)$ is a continuous function of $d_3,d_4$ which goes to $0$ as $d_3, d_4 \rightarrow 0$. The last inequality above follows from the equation \eqref{eq: micros coerc est1} as we can write $h = h^\perp +a p_1g$ which implies that $ a^2 + \beta^2 +\gamma^2 + \int \frac{\tilde{h}^2}{g} =  \int \frac{h^2}{g} =  \int \frac{ (h^{\perp})^2 }{g} + a^2$, cf. Lemma \ref{lem: orthonormal}. Then $\beta^2 +\gamma^2 =   \int \frac{ (h^{\perp})^2 }{g}  - \int \frac{h^2}{g}$, so that the claim follows.

Now we move onto $\mathcal{J}$ which is simpler as here we only want to split $h$ into its parts that are parallel and perpendicular to $g$. We have
\begin{align*}
\mathcal{J} &= a^2 + a \sqrt{\rho_g} \langle h^\perp, \mathcal{M}_\tau \rangle - a^2 - a \frac{1}{\sqrt{\rho_g}} \int \left(1+ \rho_g \frac{\mathcal{M}_\tau}{g} \right) h^\perp - \frac{1}{2} \int \left(1+ \rho_g \frac{\mathcal{M}_\tau}{g} \right) \frac{(h^\perp)^2}{g}\\
&= - \frac{1}{2} \int \left( 1+ \rho_g \frac{\mathcal{M}_\tau}{g} \right) \frac{(h^\perp)^2}{g}.
\end{align*}

Therefore
\[ \langle Ch, h \rangle \leq \frac{1}{\kappa} \left( \alpha \varphi(d_3,d_4) - \frac{1-\alpha}{2} \right) \int \frac{(h^\perp)^2}{g}. \]

Now $d_3, d_4$ change with $\alpha$, however from Prop. \ref{prop. bounds on 3 and 4 moments} we see that we do not have singularities in terms of $\alpha$ in the upper and lower bounds of $d_3,d_4$. Thus we can bound the moments uniformly over all $\alpha \in [0,1]$. This implies that there is a finite constant $M$ so that $\varphi(d_3,d_4)\leq M(\kappa, \bar{T}, \underline{T}, \bar{\tau}, \underline{\tau})<\infty$ for $M$ independent of $\alpha$. Then 
\[ \langle \mathcal{C} h, h \rangle \leq \frac{1}{\kappa} \left( \alpha M(\kappa, \bar{T}, \underline{T}, \bar{\tau}, \underline{\tau})- \frac{1-\alpha}{2} \right) \int \frac{(h^\perp)^2}{g},\]
which yields microscopic coercivity as long as $\alpha> \frac{1}{2  M(\kappa, \bar{T}, \underline{T}, \bar{\tau}, \underline{\tau})-1}$. Since we can make this bound $ M(\kappa, \bar{T}, \underline{T}, \bar{\tau}, \underline{\tau})$ large, we ensure that there is a positive constant $\lambda_m$ so the statement of the proposition holds true. 
\end{proof} 

\subsection{Macroscopic coercivity} \label{subsec: macro coerc}
We remind that if $\Pi$ is the projection onto the kernel of $\mathcal{C}$ then $\Pi h = \frac{\rho_h}{\rho_g}g$. We proceed by providing a macroscopic coercivity inequality for the antisymmetric part of the linearised operator on the null space of $\mathcal{C}$. 
\begin{prop} \label{prop:macro coercivity}
With the above notation and when the condition \eqref{eq:constr_LB on P posit} on $\alpha$ holds, for the antisymmetric operator there exists a positive constant $\lambda_M$ so that 
$$ \| \mathcal{T}\Pi h\| \geq  \lambda_M \| \Pi h\| $$
for all $h \in \mathcal{H} \cap \mathcal{D}(\mathcal{T}\Pi)$.
\end{prop}
\begin{proof} 
We first remind, cf Lemma \ref{lem:splitting sym antisym}, that the antisymmetric part $\mathcal{T}$ is given by
\begin{align*}
\mathcal{T} h &= v \partial_x h + \frac{\alpha}{2\kappa} \int \left( \mathcal{M}_{T_g}(v) - \frac{\mathcal{M}_{T_g}(u)g(x,v)}{g(x,u)} \right) \left( 1+ \frac{uv}{T_g} + \frac{1}{2} \left( \frac{u^2}{T_g}-1\right)\left( \frac{v^2}{T_g} -1 \right) \right)h(x,u) \mathrm{d}u \\
& \quad + \frac{1-\alpha}{2\kappa} \int \left( \mathcal{M}_\tau(v) - \frac{\mathcal{M}_\tau (u) g(x,v)}{g(x,u)} \right) h(x,u) \mathrm{d}u \\
& \quad - \frac{1}{2 \kappa} \left( h - \frac{\rho_g}{g} \left( \alpha \mathcal{M}_{T_g}(v) + (1-\alpha) \mathcal{M}_\tau(v) \right) h \right).
\end{align*}
From this we can explicitly compute, since $\mathcal{T}g=0$ that
\[ \mathcal{T} \Pi h = v \partial_x\left(\frac{\rho_h}{\rho_g} \right) g. \] 
Consequently with respect to the inner product defined in \eqref{eq:scalar prod},
\begin{equation} 
\begin{split}
\| \mathcal{T} \Pi h\|^2 &= \iint_{\mathbb{T} \times\mathbb{R}} v^2 g(x,v) \left(\partial_x \left(\frac{\rho_h}{\rho_g} \right) \right)^2 \mathrm{d}v \mathrm{d}x 
= \int_{\mathbb{T}}  \rho_g(x) T_g(x) \left(\partial_x \left( \frac{\rho_h}{\rho_g} \right) \right)^2 \mathrm{d}x\\ 
& \geq \frac{1}{4} \left[ \left[ \alpha \frac{\underline{T}^{1/2}}{2\pi} + (1-\alpha)\frac{\underline{\tau}^{1/2}}{2\pi}  \right]  
- \frac{1}{\kappa}  \left[ \frac{\alpha}{\kappa \underline{T}^{1/4}} + \frac{1-\alpha}{\kappa \underline{\tau}^{1/4}} \right]^2 \right]   \left[ \frac{\alpha}{\kappa \underline{T}^{1/4}} + \frac{1-\alpha}{\kappa \underline{\tau}^{1/4}} \right]^{-2} \times \\ & 
\qquad\qquad\qquad\qquad\qquad\qquad\qquad\qquad\qquad\qquad
\times  \int_{\mathbb{T}}  \rho_g(x)\left(\partial_x \left( \frac{\rho_h}{\rho_g} \right) \right)^2 \mathrm{d}x
 \end{split} 
 \end{equation}
 where we applied the lower bound on the temperature from Lemma \ref{lem:bounds on temperature}. 
Therefore we conclude the macroscopic coercivity provided that $\rho_g(x) \mathrm{d}x$ satisfies a Poincar\'{e}  inequality: Indeed then 
 $$ \| \mathcal{T} \Pi h\|^2 \geq \lambda_M
   \int_{\mathbb{T}}  \rho_g(x) \left( \frac{\rho_h}{\rho_g} \right)^2(x)  \mathrm{d}x  =  \lambda_M
   \iint_{\mathbb{T}\times\mathbb{R}}  g(x,v)^{-1}|\Pi h|^2 \mathrm{d}x\mathrm{d}v =  \lambda_M \|\Pi h \|^2$$
 where $\lambda_M = \lambda_M(\alpha, \kappa, \bar{T}, \underline{T}, \bar{\tau}, \underline{\tau} )$. 
 Now a  Poincar\'{e} inequality is true since our space is the torus and $\rho_g$ is upper and lower bounded uniformly in $x\in \mathbb{T}$: For $h \in \mathcal{H}$, i.e. $\int_{\mathbb{T}}\left(  \frac{\rho_h}{\rho_g} \right) (x) \mathrm{d}x=0$ (this is equivalent by the upper and lower bound on the density with $\int_{\mathbb{T}}\rho_h \mathrm{d}x=0$) it holds 
 $$\int_{\mathbb{T}}\left(\partial_x \left(\frac{\rho_h}{\rho_g} \right) \right)^2  \mathrm{d}x \geq \int_{\mathbb{T}} \left(\frac{\rho_h}{\rho_g} \right)^2  \mathrm{d}x. $$
 Hence macroscopic coercivity holds with the explicit constant 
 \begin{align*} 
 \lambda_M = & \frac{1}{4} \left[ \left[ \alpha \frac{\underline{T}^{1/2}}{2\pi} + (1-\alpha)\frac{\underline{\tau}^{1/2}}{2\pi}  \right]  
- \frac{1}{\kappa}  \left[ \frac{\alpha}{\kappa \underline{T}^{1/4}} + \frac{1-\alpha}{\kappa \underline{\tau}^{1/4}} \right]^2 \right]   \left[ \frac{\alpha}{\kappa \underline{T}^{1/4}} + \frac{1-\alpha}{\kappa \underline{\tau}^{1/4}} \right]^{-2}\times \\ 
& \times \left[ 1 - \Big[ \alpha \underline{T}^{-1/2}  + (1-\alpha) \underline{\tau}^{-1/2} \Big] \right]. 
\end{align*} 
 
\end{proof} 

\subsection{Boundnedness of auxiliary operators} \label{subsec:boundedness auxiliar oper}
From above we can see that 
\begin{equation} \label{eq:Pi T Pi}
\Pi \mathcal{T} \Pi h = 0
\end{equation}
 since $\int ug(x,u) \mathrm{d}u =0$. Indeed $$\Pi \mathcal{T} \Pi h = \int u \partial_x\left(\frac{\rho_h}{\rho_g} \right)(x)g(x,u) \mathrm{d}u \frac{g(x,v)}{\rho_g(x)} = \partial_x\left(\frac{\rho_h}{\rho_g} \right)(x) u_g(x) \frac{g(x,v)}{\rho_g(x)}  =0.$$

 Following the approach in \cite{DMS15}, we define the operator 
\begin{equation} 
A:= (1+(\mathcal{T} \Pi)^* \mathcal{T} \Pi )^{-1}(\mathcal{T} \Pi)^*.
\end{equation}
We will use this operator in order to define the modified entropy for the hypocoercivity. 
The final ingredients we will need is the boundedness of certain operators. 
In this subsection we record some properties of these operators. First thing to notice is that $A=\Pi A$, this is easy to see for example from the relation
 \begin{equation} \label{eq:operat A}
A f  = -\Pi \mathcal{T} f + \Pi \mathcal{T}^2 \Pi A f,\text{ for all } f \in \mathcal{H}
\end{equation} 
 which follows directly from the definition of $A$. Then taking the inner product in \eqref{eq:operat A} with $Af$ we see that \begin{align*} 
  \| Af\|^2 = \langle f, \mathcal{T} \Pi Af  \rangle &- \| \mathcal{T} \Pi A f \|^2  \leq \|(1-\Pi)f \| \| \mathcal{T} \Pi Af \| - \| \mathcal{T} \Pi A f \|^2 \text{ or } \\
 &  \| Af\|^2 +\| \mathcal{T} \Pi A f \|^2   \lesssim \frac{\|(1-\Pi)f \|^2}{2\varepsilon} + \frac{\varepsilon}{2} \|\mathcal{T} \Pi Af \|^2
 \end{align*} 
for $\varepsilon >0$, where we used \eqref{eq:Pi T Pi}. This yields the boundedness of both $A$ and $\mathcal{T} A$ for all $f \in \mathcal{H}$.

In \cite{DMS15} they require a stronger hypothesis regarding the boundedness of auxilliary operators, namely that  
\[ \|A \mathcal{T} (I-\Pi)f\| + \|A \mathcal{C} f\| \lesssim  \|(I-\Pi)f\|. \] 
It is immediate on studying the proof in this paper that it is also sufficient to show that
\[ \|A\mathfrak{L}(I-\Pi)f\| \lesssim \|(I-\Pi)f\|. \]

\begin{lemma} \label{lem: boundedn auxiliary op}
With the above notation the operators 
$$A,\ \mathcal{T} A \text{ and } A \mathfrak{L} $$
are bounded and 
\[ \|A\mathfrak{L}(I-\Pi)f\| \leq C\|(I-\Pi)f\|  \] 
for some constant $C$ depending on $\alpha, \kappa, \bar{T}, \underline{T}, \bar{\tau}, \underline{\tau}$.
\end{lemma}

\begin{proof}
The operator $A \mathfrak{L}$ is bounded if and only if $\mathfrak{L}^* A^*$ is bounded. We show this is bounded by adapting elliptic regularity style result as in \cite{DMS15} but here with some extra terms. Let us write
\[ \tilde{f} := (I+ (\mathcal{T}\Pi)^*(\mathcal{T} \Pi))^{-1} f \] and define
$m(x) = \frac{\rho_{\tilde{f}}}{\rho_g}(x)$ then

\begin{equation} 
\begin{split}
&\mathfrak{L}^* A^* f = \mathfrak{L}^*\mathcal{T} \Pi \tilde{f}  = \mathfrak{L}^* (v (\partial_x m) g(x,v) )
=  v\partial_x ( v (\partial_x m) g(x,v) )\\ 
& + \frac{\partial_x m}{\kappa} \Bigg( \alpha g(x,v) \int u \mathcal{M}_{T_g}(u) \left( 1 + \frac{uv}{T_g} + \frac{1}{2} \left[ \frac{u^2}{T_g} -1 \right] \left[ \frac{v^2}{T_g} -1 \right] 
\right) \mathrm{d}u \\ & 
\qquad\qquad\qquad\qquad\qquad\qquad\qquad+ (1-\alpha) \int \frac{ \mathcal{M}_{\tau(x)}(u) }{g(x,u)} ug(x,u) \mathrm{d}u g(x,v)
\Bigg) \\ & - 
\frac{\rho_g(x)}{g(x,v)}\left( \frac{\alpha}{\kappa} \mathcal{M}_{T_g(x)}(v) + \frac{(1-\alpha)}{\kappa} \mathcal{M}_{\tau(x)}(v) \right) vg(x,v) ( \partial_xm)\\ 
 &=v^2 (\partial^2_x m(x)) g(x,v) - \frac{1-\alpha}{\kappa} (\partial_x m(x)) v g(x,v), 
 \end{split}  
\end{equation}
where in the last line we used that $g$ is a NESS and the cancellation of the odd Maxwellian moments. 
So taking the $L_{x,v}^2(g^{-1})$ norm we have 
\begin{equation} \label{eq:bound on LA}
 \| \mathfrak{L}^* A^* f \| \leq 
\left\| T_g \sqrt{\rho_g} [\partial_x^2 m] \int v^4 g \mathrm{d} v  \right\|_{L^2(\mathbb{T})} + \frac{(1-\alpha)}{\kappa}  \left\| \sqrt{\rho_g T_g} ( \partial_x m(x)) \right\|_{L^2(\mathbb{T})} \leq C \|m\|_{H^2(\mathbb{T})}, 
\end{equation}
where $C = C(\alpha, \kappa, \bar{T}, \underline{T}, \bar{\tau}, \underline{\tau})$ and for the bounds we used the results from Lemma \ref{lem: bounds_press} and Proposition \ref{prop. bounds on 3 and 4 moments}. 
We furthermore know that 
\[ (I + (\mathcal{T} \Pi)^*(\mathcal{T} \Pi)) \tilde{f} = \tilde{f} + T_g (\partial^2_x m) g = f . \] Integrating this and dividing by $\rho_g$ gives
\[ \frac{1}{T_g(x)}m(x) - (\partial_x^2 m(x)) = \frac{1}{T_g(x)} n_g(x), \] where $n_g(x) = \frac{\rho_f}{\rho_g}(x)$. Then since $T_g$ is bounded above and below, see lemma \ref{lem:bounds on temperature}, $L^2 \to H^2$   elliptic regularity gives us that 
\begin{align*} 
&\| m\|_{H^2(\mathbb{T})} \leq C \left\| \frac{n_g}{T_g}\right\|_{L^2(\mathbb{T})}  \\ & \leq  C \left( \frac{1}{4} \left[ \left[ \alpha \frac{\underline{T}^{1/2}}{2\pi} + (1-\alpha)\frac{\underline{\tau}^{1/2}}{2\pi}  \right]  
- \frac{1}{\kappa}  \left[ \frac{\alpha}{\kappa \underline{T}^{1/4}} + \frac{1-\alpha}{\kappa \underline{\tau}^{1/4}} \right]^2 \right]   \left[ \frac{\alpha}{\kappa \underline{T}^{1/4}} + \frac{1-\alpha}{\kappa \underline{\tau}^{1/4}} \right]^{-2} \right)^{-1}  
\|n_g\|_{L^2(\mathbb{T})}.
\end{align*}

 This in turn yields from \eqref{eq:bound on LA} that
\[ \| \mathfrak{L}^* A^* f \| \leq C(\alpha, \kappa, \bar{T}, \underline{T}, \bar{\tau}, \underline{\tau}) \|n_g\|_{L^2(\mathbb{T})} 
\leq C(\alpha, \kappa, \bar{T}, \underline{T}, \bar{\tau}, \underline{\tau}) \|f\|_{L^2(g^{-1})} \] 
which concludes the lemma.
\end{proof}
\subsection{Proof of Theorem \ref{thm:stability}} 
As we have verified the main assumptions required for the abstract $L^2$ stability theorem - microscopic coercivity of $\mathcal{C}$, macroscopic coercivity of $\mathcal{T}$ and boundedness of certain operators - the stability of our linearised equation follows directly by considering, for $A$ defined in \eqref{eq:operat A}, the modified entropy 
\begin{equation} 
H(f) : = \frac{1}{2} \|f\|^2+ \varepsilon \langle Af,f \rangle
\end{equation}
where $\varepsilon>0$. One then computes the entropy dissipation 
\begin{align*}
D(f) &= -\langle \mathcal{C} f,f\rangle +
 \varepsilon \langle A \mathcal{T}\Pi f,f\rangle  + 
 \varepsilon \langle A \mathcal{C}(I-\Pi) f,f\rangle - 
 \varepsilon \langle \mathcal{T} Af,f\rangle  \\
&= -\langle \mathcal{C} f,f\rangle +
 \varepsilon \langle A \mathcal{T}\Pi f,f\rangle  + 
 \varepsilon \langle A \mathfrak{L}(I-\Pi) f,f\rangle - 
 \varepsilon \langle \mathcal{T} Af,f\rangle 
\end{align*}
The coercivity of the Dirichlet form follows thanks to Propositions \ref{prop:micro coercivity} and \ref{prop:macro coercivity} via which we treat the first two terms: 
$$  -\langle \mathcal{C} f,f\rangle +
 \varepsilon \langle A \mathcal{T}\Pi f,f\rangle \geq \lambda_m \vee \left[ \frac{\varepsilon \lambda_M}{1+\lambda_M}\right] \|f\|^2
 $$ for $\lambda_M = \lambda_M(\alpha, \kappa, \bar{T}, \underline{T}, \bar{\tau}, \underline{\tau} )$ and the same for $\lambda_m$. 
 Then up to sufficiently small $\varepsilon \in (0,1)$, (i) the new perturbed entropy is equivalent to the $L^2$ norm, i.e. $\frac{1}{2} (1-\varepsilon) \|f\|^2 \leq H(f) \leq\frac{1}{2} (1+\varepsilon) \|f\|^2$ and (ii) using also the boundedness of $A\mathfrak{L}$, cf Lemma \ref{lem: boundedn auxiliary op}, we conclude the existence of a positive constant $\lambda = \lambda(\alpha, \kappa, \bar{T}, \underline{T}, \bar{\tau}, \underline{\tau} )$ for which 
 $D(f) \geq \lambda \|f\|^2$. From this, due to the equivalence, we get from Gr\"{o}nwall an exponential relaxation of the semigroup $e^{t\mathfrak{L}}$ in $H(f)$ and therefore in $L^2$ up to prefactors.


\smallsection{Acknowledgements} The authors would like to thank Cl\'{e}ment Mouhot and Laure Saint-Raymond for interesting discussions. JE acknowledges partial support from the Leverhulme Trust, Grant ECF-2021-134 and AM support by a Huawei fellowship at IHES.

\appendix
\section{Linear stability when the temperature variation is small} \label{appendix}
\begin{thm}
Suppose that there exists a $\tau_*$ and an $\epsilon >0$ such that $|\tau(x) - \tau| < \epsilon$ and there exists a steady state $g$ such that $|T_g(x) - \tau_*| \leq C\epsilon$ for some $C$ that can depend on $\tau_*$ but not on $\epsilon$. Then the steady state $g$ is linearly stable provided $\epsilon$ is sufficiently small in terms of $\tau_*$.
\end{thm}
\begin{proof}
Let us define $$\mathfrak{L}_{\tau_*} f = - v \partial_x f + \frac{1}{\kappa} \left( \mathcal{M}_{\tau_*} \int_{-\infty}^\infty \left( 1+ \alpha \frac{uv}{\tau_*} + \frac{1}{2} \left( \frac{u^2}{\tau_*} -1 \right) \left( \frac{v^2}{\tau_*} -1 \right) \right) f(x,u) \mathrm{d} u - f \right).$$ Then we notice that using a standard $L^2$-hypocoercivity argument we can find a norm $\| \cdot \|_*$ which is equivalent to the $L^2(\mathcal{M}_{\tau_*}^{-1})$-norm for which $\mathfrak{L}_{\tau_*}$ has a spectral gap. We write
\[ \langle \mathfrak{L}_{\tau_*} f, f \rangle_* \leq - \lambda_* \|f\|_*^2 \] 
where $\lambda_*$ depends on $\tau_*, \kappa$ and the choice of norm. We also notice that 
\begin{align*} |\mathfrak{L} f - \mathfrak{L}_{\tau_*}f | & \leq \frac{1}{\kappa} \left( C(\mathcal{M}_{T_g} - \mathcal{M}_{\tau_*})\Big(\rho_f + \frac{v}{\tau_* - \epsilon} m_f \right. \\
& \quad \left.+ \frac{1+v^2}{(\tau_* - \epsilon)^2} P_f \Big) + C\mathcal{M}_{\tau_*} \left(1+ \frac{1}{\tau_* - \epsilon} \right)(1+v^2) \epsilon + (\mathcal{M}_\tau - \mathcal{M}_{\tau_*}) \rho_f\right). \end{align*}
Therefore,
\[ \| \mathfrak{L} f - \mathfrak{L}_{\tau_*}f\|_{L^2(\mathcal{M}_{\tau_*}^{-1})} \leq C(\tau_*) \left(\epsilon^2 + \|(1+v^2)(\mathcal{M}_{T_g} - \mathcal{M}_{\tau_*})\|_{L_v^2(\mathcal{M}_{\tau_*}^{-1})} +  \|\mathcal{M}_{\tau} - \mathcal{M}_{\tau_*}\|_{L_v^2(\mathcal{M}_{\tau_*}^{-1})}\right) \|f\|. \]
We can show that 
\begin{align*}
 & \int_v (1+v^2)[M_T(v)-M_{T(x)}(v)]^2M_T^{-1} dv= \int_v (1+v^2)\left[1- \frac{M_{T(x)}(v)}{M_T(v)}\right]^2M_T(v) \mathrm{d}v\\ &
 \sim \int_v (1+v^2) M_T(v) \left[ 1- \sqrt{\frac{T}{T(x)}}e^{-\frac{|v|^2}{2}(T(x)^{-1}-T^{-1})} \right]^2 \mathrm{d}v.
\end{align*}
If we write $T(x)= T+ \theta(x)$ with $\operatorname{sup}_x |\theta(x)| \le \varepsilon$, 
\begin{align*}
    \int_v (1+v^2) M_T(v) \left[ 1- \sqrt{\frac{T}{T(x)}}e^{-\frac{|v|^2}{2T}\left(\frac{1}{1-\theta(x)/T}-1  \right)} \right]^2 \mathrm{d}v \sim 
    &\int_v (1+v^2) M_T(v)\left[ 1- \sqrt{\frac{T}{T(x)}} e^{ \frac{|v|^2\theta(x)}{2} + \frac{|v|^2}{2T}} \right]^2 \mathrm{d}v  \\ & \lesssim \frac{\varepsilon^2}{4T^2} \int_v (1+v^2) M_T(v) e^{|v|^2\varepsilon}  \mathrm{d}v. 
\end{align*}
Applying the above for our temperatures $\tau(x), \tau_*, T_g,$ we have 
\[ \| \mathfrak{L} f - \mathfrak{L}_{\tau_*}f\|_{L^2(\mathcal{M}_{\tau_*}^{-1})} \leq C(\tau_*) \epsilon^2  \|f\|. \]
Therefore as our other norm is equivalent we have
\[ \| \mathfrak{L} f - \mathfrak{L}_{\tau_*}f\|_* \leq C'(\tau_*) \epsilon^2 \|f\|. \] Hence, 
\[ \frac{\mathrm{d}}{\mathrm{d}t} \|f\|^2_* \leq -(\lambda_* - C'(\tau_*) \epsilon^2) \|f\|_*^2. \] So indeed Gr\"{o}nwall implies that if $\epsilon$ is small enough we will have a positive spectral gap in the $*$-norm for the operator $\mathfrak{L}$.
\end{proof}

 \bibliographystyle{alpha}
\bibliography{bibliography}

\end{document}